\documentclass[11pt]{amsart}

\usepackage{graphicx}
\usepackage{amsfonts}
\usepackage{amscd}
\usepackage{amssymb}
\usepackage{alltt}
\usepackage{txfonts}
\usepackage{mathrsfs}

\def\op#1{\operatorname{#1}}
\def\lie#1{\mathbf{#1}}
\def\g{\lie{g}}
\def\b{\lie{b}}
\def\h{\lie{h}}
\def\k{\lie{k}}
\def\a{\lie{a}}
\def\m{\lie{m}}
\def\tx{\lie{t}}
\def\cx{\lie{c}}

\def\ring#1{\mathbb #1}
\def\inv{\op{inv}}
\def\Ad{\op{Ad}}
\def\ad{\op{ad}}
\def\Gal{\op{Gal}}
\def\val{\op{val}}
\def\cP{{\mathcal P}}
\def\Wth{{W_H\rtimes \langle\theta\rangle}}

\def\ac{\operatorname{ac}}

\def\ord{\operatorname{ord}}

\def\ord{{\operatorname {\rm ord}}}
\def\ac{{\overline{\operatorname  {\rm ac}}}}

\def\tn{\operatorname {TN}}

\def\bA{{\mathbb A}}
\def\bC{{\mathbb C}}
\def\bF{{\mathbb F}}

\def\bL{{\mathbb L}}
\def\bN{{\mathbb N}}
\def\bQ{{\mathbb Q}}
\def\bR{{\mathbb R}}
\def\bZ{{\mathbb Z}}

\def\ccF{{\mathscr F}}
\def\ccC{{C}}

\def\cA{{\mathcal A}}
\def\cB{{\mathcal B}}
\def\cC{{\mathscr C}}
\def\cD{{\mathcal D}}

\def\cF{{\mathcal F}}

\def\cL{{\mathcal L}}

\def\cO{{\mathcal O}}
\def\cP{{\mathcal P}}

\def\cR{{\mathcal R}}

\def\cU{{\mathcal U}}
\def\cV{{\mathcal V}}

\mathchardef\alphag="7C0B
\mathchardef\betag="7C0C
\mathchardef\gammag="7C0D
\mathchardef\deltag="7C0E
\mathchardef\varepsilong="7C22
\mathchardef\varphig="7C27
\mathchardef\psig="7C20
\mathchardef\zetag="7C10
\mathchardef\epsilong="7C0F
\mathchardef\rhog="7C1A
\mathchardef\taug="7C1C
\mathchardef\upsilong="7C1D
\mathchardef\iotag="7C13
\mathchardef\thetag="7C12
\mathchardef\pig="7C19
\mathchardef\sigmag="7C1B
\mathchardef\etag="7C11
\mathchardef\omegag="7C21
\mathchardef\kappag="7C14
\mathchardef\lambdag="7C15
\mathchardef\mug="7C16
\mathchardef\xig="7C18
\mathchardef\chig="7C1F
\mathchardef\nug="7C17
\mathchardef\varthetag="7C23
\mathchardef\varpig="7C24
\mathchardef\varrhog="7C25
\mathchardef\varsigmag="7C26
\mathchardef\Omegag="7C0A
\mathchardef\Thetag="7C02
\mathchardef\Sigmag="7C06
\mathchardef\Deltag="7C01
\mathchardef\Phig="7C08
\mathchardef\Gammag="7C00
\mathchardef\Psig="7C09
\mathchardef\Lambdag="7C03
\mathchardef\Xig="7C04
\mathchardef\Pig="7C05
\mathchardef\Upsilong="7C07

\def\ord{{\operatorname {\rm ord}}}

\def\gM{{\mathfrak M}}

\def\11{{\mathbf 1}}
\def\ordjac{{\operatorname {\rm ord jac}}}

\def\LO{{\cL_{\cO}}}

\def\Def{{\rm Def}}

\def\RDef{{\rm RDef}}
\def\RDefe{{\rm RDef}^{\rm exp}}

\def\jac{\operatorname{jac}}

\newtheorem{thm}[subsubsection]{Theorem}
\newtheorem{lemma}[subsubsection]{Lemma}

\newtheorem{remark}[subsubsection]{Remark}

\numberwithin{equation}{subsection}
\title{Transfer Principle for the Fundamental Lemma}

\begin{document}
\author[Raf Cluckers]{Raf Cluckers}
   \address[Raf Cluckers]{Katholieke Universiteit Leuven,
Departement wiskunde, Celestijnenlaan 200B, B-3001 Leu\-ven,
Bel\-gium. Current address: \'Ecole Normale Sup\'erieure,
D\'epartement de ma\-th\'e\-ma\-ti\-ques et applications, 45 rue
d'Ulm, 75230 Paris Cedex 05, France}
   \thanks{During the preparation of this paper R. Cluckers was
   a postdoctoral fellow of the Fund for Scientific Research - Flanders
(Belgium) (F.W.O.)} \email{cluckers@ens.fr}
\urladdr{www.dma.ens.fr/$\sim$cluckers/}
   \author[Thomas Hales]{Thomas Hales}
   \address[Thomas Hales]{Department of Mathematics\\
        University of Pittsburgh\\
        Pittsburgh, PA 15217}
   \thanks{The research of T. Hales was supported
     in part by NSF grant 0503447.  He would also like to thank
     E.N.S. for its hospitality.}
   \email{hales@pitt.edu}
      \urladdr{www.math.pitt.edu/$\sim$thales/}
   \author[Fran\c{c}ois Loeser]{Fran\c{c}ois Loeser}
   \address[Fran\c{c}ois Loeser]{{\'E}cole Normale Sup{\'e}rieure,
D{\'e}partement de math{\'e}matiques et applications, 45 rue
d'Ulm, 75230 Paris Cedex 05, France (UMR 8553 du CNRS)}
   \thanks{The project was partially supported by ANR grant 06-BLAN-0183}
   \email{Francois.Loeser@ens.fr}
   \urladdr{www.dma.ens.fr/$\sim$loeser/}
   \maketitle

\maketitle

\section*{Introduction}
The  purpose of this paper is to explain how the general transfer principle
of Cluckers and Loeser \cite{clake}\cite{clexp} may be used in the study of the fundamental lemma.
We use here the word ``transfer" in a sense originating with the
Ax-Kochen-Er{\v s}ov transfer
principle in logic.    Transfer principles in model theory are results
that transfer
theorems from one field to another.  The transfer principle of  \cite{clake}\cite{clexp}
is a general result that transfers theorems about identities of $p$-adic
integrals from
one collection of fields to others.    These general transfer principles are
reviewed in Theorems \ref{tran}
and \ref{tranexp}.
The main purpose of this article is to explain how the
identities of various fundamental lemmas fall within the scope of
these general transfer principles.  Consequently, once the fundamental
lemma has been
established for one collection of fields (for example, fields of
positive characteristic),
it is also valid for others (fields of characteristic zero).   Precise
statements
appear in Theorems \ref{mt1} (for the fundamental lemma),  \ref{mt2} (for the weighted
fundamental lemma),
and Section~\ref{jyf}  (for the Jacquet-Ye relative fundamental lemma).

In an unfortunate clash of terminology, the word ``transfer" in the
context of the fundamental lemma has come to mean
the matching of smooth functions on a reductive
group with those on an endoscopic group.  We have nothing to say about
transfer in
that sense.  For example, Waldspurger's article from 1997 ``Le lemme fondamental implique le transfert" is completely unrelated
(insofar as it is possible for two articles on the fundamental lemma to
be unrelated).

The intended audience for this paper being that of  mathematicians
working in the areas of automorphic forms and representation theory,
we tried our best to make all definitions and statements from other fields
that are used in this paper understandable without any prerequisite.
In particular, we start the paper by  giving a quick presentation of first-order languages and the Denef-Pas language
and an overview   on motivic constructible functions and their integration according to \cite{cl},
before stating the general  transfer principle.
The bulk of the paper consists in proving the definability of the various data occurring in the fundamental
lemma. Once this is achieved,  it is not difficult to deduce our main
 result in \ref{themainth}, stating  that the transfer principle
holds for the  integrals occurring in the fundamental lemma, which is of special interest in view of the recent
advances by  Laumon and Ng\^o
\cite{ln}
and
Ng\^o \cite{ngo}.

Other results concerning the transfer
principle for the fundamental lemma appear
in \cite{cun}, \cite{wald}, \cite{w-weighted}.

\medskip

\textit{
 We thank
Michael Harris for inviting an expository
paper on this topic for the book
he is editing on
``Stabilisation de la formule des traces,
vari\'et\'es de Shimura,
et applications arithm\'etiques''
available at
\newline
http://www.institut.math.jussieu.fr/projets/fa/bp0.html
}

\section{First order languages and the Denef-Pas language}\label{sec1}

\subsection{Languages}A (first order) language $L$ consists of an enumerable infinite set of symbols of variables
$\cV = \{v_0, \cdots, v_n, \cdots \}$, logical symbols
$\neg$ (negation), $\wedge$ (or), $\vee$ (and), $\implies$, $\iff$, $\forall$ and
$\exists$,
together with two suites of sets $\cF_{n}$ and
$\cR_{n}$, $n \in \bN$.
Elements of  $\cF_n$ will be symbols of $n$-ary functions, elements of
 $\cR_n$ symbols of $n$-ary relations.
A $0$-ary function symbol will be called a constant symbol.
The language $L$ consists of the union of these sets of symbols.

\subsection{Terms}The set $T (L)$ of terms of the language $L$
is defined in the following way:
variable symbols and constant symbols are terms and if
 $f$ belongs to  $\cF_n$ and
 $t_1$, \dots, $t_n$ are terms, then $f (t_1, \cdots, t_n)$ is also a term.
 A more formal definition of  $T (L)$ is to view it as a subset of the set of  finite words
 on $L$. For instance,
 to the word $f t_1 \dots t_n$ corresponds the  term $f (t_1, \cdots, t_n)$.
One defines the weight of a function as its arity $- 1$ and we
 give to symbols of variables the weight $-1$. The weight of a finite word on $L$ is the sum of the weights of its symbols. Then terms correspond exactly to finite words on
 variable and function symbols
 of total weight $-1$ whose strict initial segments are of weight $\geq 0$,
 when nonempty.
If  $t$ is a term one writes  $t = t [w_0, \dots, w_n]$
to mean that all variables occurring in $t$ belong to the $w_i$'s.


 \subsection{Formulas}

An atomic formula is an expression of the form
$R (t_1, \dots, t_n)$ with  $R$ an $n$-ary relation symbol and   $t_i$ terms.
The set of formulas in $L$
is the smallest set containing atomic formulas and such that   if
$M$ and $N$  are formulas then
$\neg M$, $(M \wedge N)$, $(M \vee N)$, $(M \implies N)$, $(M \iff N)$,
$\forall v_n M$ et $\exists v_n M$ are formulas.
Formulas may also be defined as certain finite words on $L$.
(Parentheses are just a way to rewrite terms and formulas in a more
handy way, as opposed to writing them as finite words on $L$).

 Let $v$ be a variable symbol occurring in a formula $F$. If $F$ is atomic we say all occurrences of $v$ in
 $F$ are free. If $F = \neg G$ the free occurrences of  $v$ in $F$
are those in $G$.
Free occurrences of  $v$ in  $(F\alpha G)$ are those in  $F$ and those in  $G$
where $\alpha$ is either $\neg$, $\wedge$, $\vee$, $\implies$, or
$\iff$.
If $F = \forall w G$ or $\exists w G$ with $w \not=v$,
free occurrences of  $v$ in  $F$
are those in  $G$.
When  $v = w$, no occurrence of  $v$ in $F$ free.
Non free occurrences of a variable are called bounded.
Free variables in a formula  $F$ are those having at least one free occurrence.
A sentence is a formula with no free variable.

We write $F [w_0, \dots, w_n]$  if all free variables in  $F$  belong to the
 $w_i$'s (supposed to be distinct).

\subsection{Interpretation in a structure}
Let $L$ be a language. An $L$-structure $\gM$ is a set  $M$ endowed
for every $n$-ary function symbol $f$ in $L$
with  a function
$f^{\gM} : M^n \rightarrow M$ and for every $n$-ary relation $R$ with  a subset
$R^{\gM}$  of $M^n$.

 If $t [w_0, \dots, w_n]$ is a term
and $a_0$, \dots, $a_n$ belong to  $M$,
 we denote by  $t [a_0, \dots, a_n]$ the interpretation of $t$ in
 $M$ defined by interpreting  $w_i$ by $a_i$.
 Namely, the interpretation of the term $w_i$ is
$a_i$, that of the constant symbol $c$  is $c^{\gM}$, and that of
the term $f (t_1, \dots t_r)$ is
 $f^{\gM} (t_1 [a_0, \dots, a_n], \dots, t_r [a_0, \dots, a_n])$.

Similarly, if  $F [w_1, \dots, w_n]$ is a formula and  $a_1$, \dots, $a_n$ belong to  $M$,
there is a natural way to interpret $w_i$ as  $a_i$ in $M$
in the formula $F$, yielding a statement $F [a_1, \dots, a_n]$ about
the tuple $(a_1, \dots, a_n)$ in $M$ which is either true or false.
One says that
 $(a_1, \dots, a_n)$ satisfies $F$ in $\gM$, and writes
 $\gM \models F [a_1, \dots, a_n]$,
if the statement $F [a_1, \dots, a_n]$ obtained by  interpreting  $w_i$ as
$a_i$ is satisfied (true) in $M$.
 For instance,
when
 $F =R (t_1, \dots, t_r)$, then
  $\gM \models F [a_1, \dots, a_n]$ if and only if
  $R^{\gM} (t_1 [a_1, \dots, a_n], \dots, t_r  [a_1, \dots, a_n])$
holds, and
for a formula $F[w_0, \dots, w_n]$, one has
$\gM \models (\forall w_0 F) [a_1, \dots, a_n]$,   if
  and only if
  for every
  $a$ in $M$,
 $\gM \models  F [a, a_1, \dots, a_n]$.

When the language contains the binary relation symbol of equality $=$,
one usually assumes $L$-structures to be equalitarian, that is, that the relation
$=^{\gM}$ coincides with the equality relation on the set $M$.
From now on we shall denote in the same way symbols $f$ and $R$ and their interpretation $f^{\gM}$ and $R^{\gM}$. In particular we may  identify constant symbols
and their interpretation in $M$ which are elements of $M$.

We shall also be lax with the names of variables and allow other names,
like $x_i$, $x$,
$y$.

\subsection{Some examples}
Let us give some examples of languages we shall use in this paper.
It is enough to  give the  symbols which are not variables nor logical.

For the  language of abelian groups these symbols consist of
the constant symbol $0$, the two binary function symbols $+$, $-$  and equality.
The language of ordered abelian groups is obtained by adding a binary relation symbol
$<$, the language of rings by adding symbols $1$ and  $\cdot$ (with the obvious arity).
Hence a structure for the ring language is just a set with interpretations for the symbols
$0$, $1$, $+$, $-$ and $\cdot$. This set does not have to be a ring (but it will be in all cases we shall consider).

If $S$ is a set, then by the ring language with coefficients in $S$, we mean that
we add $S$ to the set of constants in the language.
For instance any ring containing $S$ will be a structure for that language.


Note  that there is a sentence $\varphi$ in the ring language such that a structure
$\gM$ satisfies $\varphi$ if and only of $\gM$ is a field, namely the conjunction of the
field axioms (there is a finite number of such axioms, each expressible by a sentence
in the ring language). On the other hand, one can show there is no
sentence in the ring language expressing  for a field to be algebraically closed.
Of course, given a natural number $n > 0$, there is a sentence expressing that every degree $n$ polynomial has a root namely
\begin{equation}
\forall a_0 \forall a_1 \cdots \forall a_n \exists x
(a_0=0\vee a_0 x^n + a_1 x^{n - 1} + \cdots + a_n = 0).
\end{equation}
Note that here $x^i$ is an abbreviation for $x \cdot x \cdot x \cdots x$ ($i$ times).
It is important to notice that we are not allowed  to quantify over $n$ here,
since it  does not correspond to a variable in the structure we are considering.

\subsection{The Denef-Pas language}
We shall need a slight generalization of the notion of language, that of many sorted languages (in fact $3$-sorted language). In a  $3$-sorted language we have $3$
sorts of variables symbols, and for relation and function symbols one should specify
the type of the variables involved and for functions also the type of the value of $f$.
A structure for a $3$-sorted language
will  consist of $3$ sets $M_1$, $M_2$ and $M_3$ together with interpretations
of the non logical, non variable symbols. For instance if $f$ is a binary function symbol,
with first variable of type $2$, second variable of type $3$, and value of type $3$,
its interpretation will be a function $M_2 \times M_3 \rightarrow M_3$.

Let us fix a field $k$ of characteristic $0$ and consider
the following $3$-sorted language, the Denef-Pas
language $\cL_{\text{DP}}$.
The $3$ sorts are respectively called
the valued field sort, the residue field sort,  and the value group sort.
The language will consist of the disjoint union of
the language of rings with coefficients in
$k ((t))$ restricted to the valued field sort, of
the language of rings with coefficients in
$k$ restricted to the residue field sort and of
the language
of {ordered groups} restricted to the value group sort, together with
two additional symbols of unary functions {$\ac$}  and {$\ord$} from the valued field sort to the
residue field and valued groups sort, respectively.
[In fact, the definition we give here  is different from that  in \cite{cl}, where for the value group sort in
$\cL_{\text{DP}}$ symbols  $\equiv_n$
for equivalence relation modulo $n$,  $n >1$ in $\bN$,
are added, but since this does not change the category of definable objects,
this change has no consequence on our statements.]

An example of an $\cL_{\text{DP}}$-structure is $(k ((t)), k, \bZ)$
with $\ac$ interpreted as the function
$\ac : k((t)) \rightarrow k$ assigning to a series its first nonzero coefficient if not zero, zero otherwise, and $\ord$ interpreted as the valuation function
$\ord : k((t)) \setminus \{0\} \rightarrow \bZ$. (There is a minor divergence here,
easily fixed,  
 since
$\ord \, 0$ is not defined.)
More generally, for any field $K$ containing $k$,
$(K ((t)), K, \bZ)$ is naturally an $\cL_{\text{DP}}$-structure.
For instance
$\ac (x^2 + (1 + t^3) y) - 5z^3$ and $\ord (x^2 + (1 + t^3) y) - 2 w + 1$
are terms in  $\cL_{\text{DP}}$,
$\forall x \exists z \neg (\ac (x^2 + (1 + t^3) y) - 5z^3 = 0)$ and
$\forall x \exists w ( \ord (x^2 + (1 + t^3) y) = 2 w + 1)$ are formulas.

\section{Integration of constructible motivic functions}\label{sec2}

\subsection{The category of definable objects}
Let $\varphi$ be a formula in the language $\cL_{\text{DP}}$  having
respectively $m$, $n$, and $r$
free variables in the various sorts. To such  a formula $\varphi$
we assign, for every field $K$ containing $k$, the subset
$h_{\varphi} (K)$ of $K ((t))^m \times K^n \times \bZ^r$
consisting of all points satisfying $\varphi$,
that is,
\begin{equation}\label{gugu}
h_{\varphi} (K) := \Bigl\{(x, \xi, \eta) \in K ((t))^m \times K^n \times \bZ^r  ;
(K((t), K, \bZ) \models \varphi (x, \xi, \eta)\Bigr\}.
\end{equation}
We shall call the
datum of such subsets for all $K$ definable (sub)assign\-ments. In
analogy with algebraic geometry, where the emphasis is not put
anymore on equations but on the functors they define, we consider
instead of formulas the corresponding subassignments (note $K
\mapsto h_{\varphi} (K)$ is in general not a functor).

More precisely,  let $F : \mathcal{C} \rightarrow \text{Set}$ be a functor from
a category  $ \mathcal{C}$ to the category of sets. By a {subassignment} $h$
of  $F$ we mean the datum, for every  object $C$ of $ \mathcal{C}$, of a
subset $h (C)$ of $F (C)$.
Most of the standard operations of
elementary set theory extend trivially to subassignments. For
instance, given subassignments $h$ and $h'$ of the same functor,
one defines subassignments $h \cup h'$, $h \cap h'$ and the
relation $h \subset h'$, etc.
When  $h \subset h'$ we say  $h$ is
a subassignment of $h'$.
A morphism $f : h \rightarrow h'$ between
subassignments of functors $F_1$ and $F_2$ consists of the datum
for every object $C$ of a map
\begin{equation}f (C) : h (C) \rightarrow h'(C).\end{equation}
The graph of $f$ is the subassignment \begin{equation}C \mapsto \text{graph} (f
(C))\end{equation}
of $F_1 \times F_2$.
Let
$k$ be a field and consider the category $F_k$ of fields
containing $k$.  (To avoid any
set-theoretical issue, we  fix a Grothendieck universe $\cU$
containing $k$ and we define ${F}_{k}$ as the small category
of all fields in $\cU$ containing $k$.)

We denote by $h [m, n, r]$ the functor $F_k
\rightarrow \text{Ens}$ given by  \begin{equation}h [m, n, r] (K) = K ((t))^m
\times K^n \times \bZ^r.\end{equation}
 In particular, $h [0, 0, 0]$ assigns the
one point set to every $K$.
We sometimes write $\bZ^r$ for $h[0,0,r]$.
Thus, to  any formula $\varphi$ in
$\cL_{\text{DP}}$  having respectively $m$, $n$,
and $r$ free variables in the various sorts, corresponds a
subassignment $h_{\varphi}$ of $h [m, n, r]$
by (\ref{gugu}). Such
subassignments are called  definable subassignments.

We denote by $\Def_k$
the   category whose objects are definable subassignments of some
$h [m, n, r]$, morphisms in $\Def_k$ being morphisms of
subassignments $f : h \rightarrow h'$ with  $h$ and $h'$ definable
subassignments of $h [m, n, r]$ and  $h [m', n', r']$ respectively
such that {the graph of $f$ is a definable subassignment}.
Note that
$h [0, 0, 0]$ is the final object in this category.

\subsection{First sketch of construction of the motivic measure}The  construction in \cite{cl} relies in an essential way
on a cell decomposition theorem due to Denef and Pas  \cite{Pas}.
Let us introduce the notion of cells.
Fix
coordinates $x = (x',z)$ on  $h [n + 1, m, r]$
with $x'$ running over
$h [n , m, r]$ and $z$ over $h [1, 0, 0]$.

A {0-cell}   in $h [n + 1, m, r]$ is a definable subassignment
$Z^0_{A}$ defined by
\begin{equation}
x' \in A \quad \text{and} \quad  z = c (x')
\end{equation}
with $A$ a definable subassignment of
$h [n , m, r]$ and $c$ a morphism
$A \rightarrow
h [1, 0, 0]$.

A {1-cell}   in $h [n + 1, m, r]$ is a definable subassignment
$Z^1_{A}$ defined by
\begin{equation}
x' \in A, \quad  \ac (z - c (x')) = \xi (x') \quad \text{and} \quad
\ord (z - c (x')) = \alpha (x')
\end{equation}
with $A$ a definable subassignment of
$h [n , m, r]$, $c$, $\xi$ and $\alpha$ morphisms from
$A$ to
$h [1, 0, 0]$, $h [0, 1, 0] \setminus \{0\}$ and $h [0, 0, 1]$,  respectively.

The  Denef-Pas Cell Decomposition Theorem states that, after adding a finite number of auxiliary parameters in the residue field and value group sorts, every definable subassignment becomes a finite disjoint union of cells:

\begin{thm} [Denef-Pas Cell Decomposition \cite{Pas}]\label{cellth}
Let $A$ be a definable subassignment $h [n + 1, m, r]$. After adding a finite number of auxiliary parameters in the residue field and value group sorts, $A$ is
a finite disjoint union of cells,
 that is,
there exists an embedding
\begin{equation}
\lambda : h [n + 1, m, r] \longrightarrow h [n + 1, m + m', r  + r']
\end{equation}
such that the composition of $\lambda$ with the projection to $ h [n + 1, m, r]$
is the identity on $A$ and such that $\lambda (A) $ is
a finite disjoint union of cells.
\end{thm}

The construction of the motivic measure
$\mu (A)$ for a definable subassignment $A$ of
$h [n, m, r]$ goes roughly as follows (more details will be given in
\ref{mm}).
  The cell decomposition theorem expresses a definable
  subassignment (in $h[n,m,r]$, $n > 0$) as a disjoint union of cells.  The measure
  of a definable subassignment is defined to be the sum of the measures of its
  cells.  In turn,  a cell in $h[n,m,r]$ is expressed in terms
  of a definable subassignment $B$ in
  $h[n-1, m', r']$ and auxiliary data.   The measure of a cell can then be defined
  recursively in terms of the "smaller" definable subassignment $B$.   The
  base case of the recursive definition is $n=0$ (with larger values of
$m$ and $r$).
  When $n=0$,
  one may consider the {counting measure} on the
$\bZ^r$-factor and the  {tautological measure} on the
$h [0, m, 0]$-factor, assigning to a definable subassignment of $h [0, m, 0]$
its class in $\cC_+ (\text{point})$ (a semiring defined in the next section).
The whole point is to check that the construction is {invariant under permutations}
of valued field coordinates. This is the more difficult part of the proof and is essentially equivalent to a form of the
{motivic Fubini Theorem}.

\subsection{Constructible functions}\label{2.3}
For $X$ in $\Def_k$ we now define
the semiring $\cC_+(X)$, resp ring $\cC(X)$,
of non negative constructible motivic functions, resp.
constructible motivic functions.

One considers the category $\Def_X$ whose objects are morphisms
$Y \rightarrow X$ in $\Def_k$, morphisms being morphisms
$Y \rightarrow Y'$ compatible with the projections to $X$.
Of interest to us will be
the subcategory $\RDef_X$ of $\Def_X$ whose
objects are definable subassignments of  $X \times h [0, n,
0]$, for variable $n$.
We shall denote by
$SK_0 (\RDef_X)$ the free abelian {semigroup} on isomorphism classes
of objects of $\RDef_X$ modulo the additivity relation
\begin{equation}
[Y]+ [Y'] = [Y \cup Y'] +[Y \cap Y'].
\end{equation}
It is endowed with a natural semiring structure.
One defines similarly the Grothen\-dieck ring
$K_0 (\RDef_X)$,  which is the ring associated to the semiring
$SK_0 (\RDef_X)$.
Proceeding this way, we only defined ``half" of
$\cC_+ (X)$ and $\cC (X)$.

To get the remaining ``half"  one considers
the ring
\begin{equation}
\bA := \bZ \Bigl[\bL, \bL^{-1}, \Bigl(\frac{1}{1 - \bL^{-i}}\Bigr)_{i > 0} \Bigr].
\end{equation}
\medskip
For $q$ a real number $>1$,
we denote by  $\vartheta_q$ the ring morphism
\begin{equation}
\vartheta_q : \bA \longrightarrow \bR
\end{equation}
{sending $\bL$ to $q$}
and we consider
the semiring
\begin{equation}
\bA_+ := \Bigl\{x \in \bA \mid \vartheta_q (x) \geq 0, \forall q > 1\Bigr\}.
\end{equation}
We denote by $\vert X\vert$ the set
of points of $X$, that is,  the set of pairs $(x_0, K)$
with $K$ in $F_k$
and $x_0 \in X (K)$,
and we
consider the subring $\cP (X)$ of the ring of functions $\vert
X \vert \rightarrow \bA$ generated by constants in $\bA$ and by all
functions $\alpha$ and $\bL^\alpha$ with $\alpha: X \rightarrow
\bZ$ definable morphisms.
We define $\cP_+ (X)$ as the semiring of functions in  $\cP
(X)$ taking their values in $\bA_+$.
These are the second ``halves".

To glue the two ``halves", one proceed as follows.
One denotes by $\bL - 1$  the class
of the subassignment $x \not= 0$ of $X \times h [0, 1, 0]$ in
$SK_0 (\RDef_X)$, resp $K_0 (\RDef_X)$.
One considers  the subring $\cP^0 (X)$ of  $\cP (X)
$, resp. the subsemiring $\cP^0_+ (X)$ of  $\cP_+ (X) $, generated
by functions of the form ${\bf 1}_Y$ with $Y$ a definable
subassignment of $X$ (that is, ${\bf 1}_Y$ is the characteristic function of $Y$), and by the constant function $\bL - 1$.
We have canonical morphisms $\cP^0 (X) \rightarrow K_0 (\RDef_X)$
and $\cP^0_+ (X) \rightarrow SK_0 (\RDef_X)$.
We may now set
\begin{equation*}
\cC_+ (X) = SK_0 (\RDef_X) \otimes_{\cP^0_+ (X)} \cP_+ (X)\end{equation*}
and
\begin{equation*}\cC (X) = K_0 (\RDef_X) \otimes_{\cP^0 (X)} \cP (X).
\end{equation*}

There are some easy functorialities.
For
every morphism $f : S \rightarrow S'$, there is a natural pullback
by $f^* : SK_0
(\RDef_{S'}) \rightarrow SK_0 (\RDef_S)$ induced by
the fiber product. If $f : S \rightarrow S'$ is a morphism
in  $\RDef_{S'}$,
composition with  $f$ induces a morphism $f_! : SK_0 (\RDef_S)
\rightarrow SK_0 (\RDef_{S'})$. Similar constructions apply to $K_0$.
If $f : S \rightarrow S'$ is a morphism in $\Def_k$, one shows in
\cite{cl} that the morphism $f^*$ may naturally be extended to a
morphism
\begin{equation}f^* : \cC_+ (S')
\longrightarrow \cC_+ (S).
\end{equation} If, furthermore, $f$ is a morphism in
$\RDef_{S'}$, one shows that the morphism $f_!$ may naturally be
extended to
\begin{equation}\label{fc}f_!
: \cC_+ (S) \longrightarrow \cC_+(S').\end{equation}
Similar functorialities exist for  $\cC$.

\subsection{Taking care of dimensions}\label{dim}
In fact, we shall need to consider not only functions as we just
defined, but functions defined almost everywhere in a given
dimension, that we call $\ccF$unctions. (Note the calligraphic
capital in $\ccF$unctions.)

We start by defining a good notion of dimension for  objects of
$\Def_k$. Heuristically, that dimension corresponds to counting
the dimension only in the valued field variables, without taking
in account the remaining variables. More precisely, to any
algebraic subvariety $Z$ of $\bA^m_{k ((t))}$ we assign the
definable subassignment $h_Z$ of $h [m, 0, 0]$ given by  $h_Z (K)
= Z (K ((t)))$. The Zariski closure of a subassignment $S$ of $h
[m, 0, 0]$ is the intersection $W$ of all algebraic subvarieties
$Z$ of $\bA^m_{k ((t))}$ such that $S \subset h_Z$. We define the
dimension of $S$ as $\dim S := \dim W$. In the general case, when
$S$ is a subassignment of $h [m, n, r]$, we define  $\dim S$ as
the  dimension of the image of  $S$ under the  projection $h [m,
n, r] \rightarrow h [m, 0, 0]$.
One can prove that isomorphic objects of $\Def_k$
have the same dimension.

For every non negative integer $d$, we denote by $\cC_+^{\leq d}
(S)$ the ideal of  $\cC_+ (S)$ generated by functions $\11_{Z}$ with
$Z$ definable subassignments of $S$ with $\dim Z \leq d$. We set
$\ccC_+ (S) = \oplus_d  \ccC^d_+ (S)$ with $\ccC^d_+ (S) := \cC_+^{\leq d}
(S) / \cC_+^{\leq d-1} (S)$. It is a graded abelian semigroup, and
also a $\cC_+ (S)$-semimodule. Elements of $\ccC_+ (S) $ are called
positive constructible $\ccF$unctions on  $S$. If $\varphi$ is a
function lying in $\cC_+^{\leq d} (S)$ but not in  $\cC_+^{\leq d -
1} (S)$, we denote by  $[\varphi]$ its  image in $\ccC^d_+ (S)$. One
defines similarly $\ccC (S)$ from $\cC (S)$.

One of the reasons why we consider functions which are defined
almost everywhere originates in the differentiation of functions
with respect to the valued field variables: one may show that a
definable function $c : S\subset h[m,n,r] \rightarrow h [1, 0, 0]$
is differentiable (in fact even analytic) outside a definable
subassignment of $S$ of dimension $< {\rm dim} S$.
In particular, if $f : S \rightarrow S'$ is an isomorphism in
$\Def_k$, one may define
a function $\ordjac f$, the order of the jacobian of $f$, which is defined
almost everywhere and is equal almost everywhere to a definable
function, so we may define $\bL^{- \ordjac f}$ in $\ccC_+^d (S)$ when
$S$ is of dimension $d$.

\subsection{Push-forward}\label{gm}

Let $k$ be a field of characteristic zero. Given $S$ in  $\Def_k$,
we define $S$-integrable $\ccF$unctions and construct pushforward
morphisms for these:

\begin{thm} [Cluckers-Loeser \cite{cl}]\label{mt}Let $k$ be a field of characteristic zero and
let $S$ be in  $\Def_k$. There exists a unique functor $Z \mapsto
{\rm I}_S \ccC_+(Z)$ from  $\Def_S$ to the category of abelian
semigroups, the functor of  $S$-integrable $\ccF$unctions, assigning
to every morphism $f : Z \rightarrow Y$ in $\Def_S$ a morphism  $f_!
: {\rm I}_S \ccC_+(Z) \rightarrow {\rm I}_S \ccC_+(Y)$ such that for every
$Z$ in $\Def_S$, ${\rm I}_S \ccC_+(Z)$ is a graded subsemigroup of $\ccC_+
(Z)$  and  ${\rm I}_S \ccC_+(S) = \ccC_+ (S)$, satisfying the following
list of axioms (A1)-(A8).
\end{thm}

\noindent {\bf {\rm (A1a)} (Naturality)}

\noindent If  $S \rightarrow S'$ is a morphism in  $\Def_k$ and $Z$
is an object in $\Def_S$, then any $S'$-integrable $\ccF$unction
$\varphi$ in $\ccC_+(Z)$ is $S$-integrable and  $f_! (\varphi)$ is the
same, considered in  ${\rm I}_{S'}$ or in ${\rm I}_{S}$.

\bigskip

\noindent {\bf {\rm (A1b)} (Fubini)}

\noindent A positive $\ccF$unction  $\varphi$ on $Z$ is
$S$-integrable if and only if it is $Y$-integrable and $f_!
(\varphi)$ is $S$-integrable.

\bigskip

\noindent {\bf {\rm (A2)} (Disjoint union)}

\noindent If $Z$ is the disjoint union of two definable
subassignments $Z_1$ and  $Z_2$, then the isomorphism $\ccC_+ (Z)
\simeq \ccC_+ (Z_1) \oplus \ccC_+ (Z_2)$ induces an  isomorphism ${\rm
I}_S \ccC_+ (Z) \simeq {\rm I}_S \ccC_+ (Z_1) \oplus {\rm I}_S \ccC_+
(Z_2)$, under which  $f_! = f_{|Z_1 !} \oplus f_{|Z_2 !}$.

\bigskip

\noindent {\bf {\rm (A3)} (Projection formula)}

\noindent For every $\alpha$ in  $\cC_+ (Y)$ and every  $\beta$
in ${\rm I}_S \ccC_+ (Z)$, $\alpha f_! (\beta)$ is $S$-integrable if
and only if $f^* (\alpha) \beta$ is, and then $f_! (  f^* (\alpha)
\beta)  = \alpha f_! (\beta)$.

\bigskip

\noindent {\bf {\rm (A4)} (Inclusions)}

\noindent If  $i : Z \hookrightarrow Z'$ is the inclusion of
definable subassignments of the same object of $\Def_S$, then 
$i_!$ is
induced by extension by zero outside  $Z$ and sends 
${\rm I}_S \ccC_+ (Z)$ injectively to ${\rm I}_S \ccC_+ (Z')$.

\bigskip

\noindent {\bf {\rm (A5)} (Integration along
residue field variables)}

\noindent Let $Y$ be an object of  $\Def_S$ and denote by $\pi$ the
projection $Y [0, n, 0] \rightarrow Y$. A $\ccF$unction $[\varphi]$
in $\ccC_+ (Y [0, n, 0])$ is $S$-integrable if and only if, with
notations of \ref{fc}, $[\pi_! (\varphi)]$ is $S$-integrable and
then $\pi_! ([\varphi]) = [\pi_! (\varphi)]$.

Basically this  axiom means that
integrating with respect to
variables in the residue field
just amounts
to taking the  pushforward induced by composition at the level
of Grothendieck semirings.

\bigskip

\noindent {\bf {\rm (A6)} (Integration along
$\bZ$-variables)}
Basically, integration along $\bZ$-variables corresponds to summing
over the integers, but to state precisely (A6), we need to
perform some preliminary  constructions.

Let us consider a function in $\varphi$ in
$\cP (S [0, 0, r])$, hence
$\varphi$ is a function $\vert S \vert \times \bZ^r \rightarrow A$.
We shall say  $\varphi$ is
$S$-integrable if for every
$q > 1$ and every  $x$ in $\vert S\vert$,
the series $\sum_{i \in \bZ^r} \vartheta_q (\varphi (x, i))$
is summable. One proves that
if $\varphi$ is
$S$-integrable there exists a unique function $\mu_S (\varphi)$
in $\cP (S)$ such that
$\vartheta_q (\mu_S (\varphi) (x))$ is equal to the sum of the previous series
for all $q > 1$ and all  $x$ in $\vert S\vert$.
We denote by  ${\rm I}_S \cP_+ (S[0, 0, r])$
the set of  $S$-integrable functions in
$ \cP_+ (S[0, 0, r])$
and we set
\begin{equation}
{\rm I}_S \cC_+ (S[0, 0, r]) = \cC_+ (S)  \otimes_{\cP_+ (S)} {\rm I}_S \cP_+ (S[0, 0, r]).
\end{equation}
Hence ${\rm I}_S \cP_+ (S[0, 0, r])$
is a sub-$\cC_+ (S)$-semimodule of
$ \cC_+ (S[0, 0, r])$ and $\mu_S$ may be extended by tensoring to
\begin{equation}
\mu_S : {\rm I}_S \cC_+ (S[0, 0, r]) \rightarrow \cC_+ (S).
\end{equation}

Now we can state (A6):

\noindent Let $Y$ be an object of $\Def_S$ and denote by $\pi$ the
projection $Y [0, 0, r] \rightarrow Y$. A $\ccF$unction $[\varphi]$
in $\ccC_+ (Y [0, 0, r])$ is  $S$-integrable if and only if there
exists $\varphi'$ in $\cC_+ (Y [0, 0, r])$ with  $[\varphi'] =
[\varphi]$ which is  $Y$-integrable in the previous sense and such
that $[\mu_Y (\varphi')]$ is $S$-integrable. We then have $\pi_!
([\varphi]) = [\mu_Y (\varphi')]$.

\bigskip

\noindent {\bf {\rm (A7)} (Volume of balls)}
It is natural to require  (by analogy with the $p$-adic case)
that the volume of a ball
$\{ z \in h [1, 0, 0] \mid \ac (z -c ) = \alpha , \ac (z -c ) = \xi\}$,
with $\alpha$ in $\bZ$, $c$ in $k ((t))$ and $\xi$ non zero in $k$,
should be $\bL^{- \alpha -1}$.
(A7) is a relative version of that statement:

\noindent Let  $Y$ be an object in $\Def_S$ and let  $Z$ be the
definable subassignment of $Y [1, 0, 0] $ defined by  $\ord (z - c
(y)) = \alpha (y)$ and $\ac (z - c (y)) = \xi (y)$, with  $z$ the
coordinate on the $\bA^1_{k ((t))}$-factor and  $\alpha, \xi, c$
definable functions on $Y$ with values respectively in $\bZ$,
$h[0, 1, 0] \setminus \{0\}$, and  $h[1, 0, 0]$. We denote by  $f
: Z \rightarrow Y$ the morphism induced by  projection. Then
$[\11_Z]$ is $S$-integrable if and only if $\bL^{-\alpha - 1}
[\11_Y]$ is, and then $f_! ([\11_Z]) = \bL^{-\alpha - 1} [\11_Y]$.

\bigskip

\noindent {\bf {\rm (A8)} (Graphs)} This last axiom expresses the
pushforward for graph projections.
It relates volume and differentials and is a special
case of the change of variables Theorem \ref{mcvf}.

\noindent Let $Y$ be in  $\Def_S$ and let $Z$ be the definable
subassignment of $Y [1, 0, 0] $ defined by  $z - c (y) = 0$ with
$z$ the coordinate on the $\bA^1_{k ((t))}$-factor and  $c$ a
morphism  $Y \rightarrow h[1, 0, 0]$. We denote by  $f : Z
\rightarrow Y$ the morphism induced by projection. Then $[\11_Z]$
is $S$-integrable if and only if $\bL^{(\ordjac f) \circ f^{-1}}$ is,
and then  $f_! ([\11_Z]) = \bL^{(\ordjac f) \circ f^{-1}}$.

\bigskip

Once Theorem \ref{mt} is proved, one may proceed as follows to
extend the constructions from $\ccC_+$ to $\ccC$. One defines ${\rm I}_S \ccC
(Z)$ as the subgroup of $\ccC (Z)$ generated by the image of  ${\rm
I}_S \ccC_+ (Z)$. One shows that if  $f : Z \rightarrow Y$ is a
morphism in $\Def_S$, the  morphism $f_! : {\rm I}_S \ccC_+(Z)
\rightarrow {\rm I}_S \ccC_+(Y)$ has a natural extension
\begin{equation}\label{IC(Z)}
f_! : {\rm I}_S \ccC(Z) \rightarrow {\rm I}_S \ccC(Y).
\end{equation}

\medskip

The proof of Theorem \ref{mt} is quite long and involved. In a
nutshell, the basic idea is the following. Integration along
residue field variables is controlled by (A5) and integration
along $\bZ$-variables by (A6). Integration along valued field
variables is constructed one variable after the other. To
integrate with respect to one valued field variable, one may,
using (a variant of) the cell decomposition Theorem \ref{cellth}
(at the cost of introducing additional new residue field and
$\bZ$-variables), reduce to the case of cells which is covered by
(A7) and (A8).
An important step is to show that this is independent of the choice of a
cell decomposition. When one integrates with respect to more than
one valued field variable (one after the other) it is crucial to
show that it is independent of the order of the variables, for
which we use a notion of bicells.

\subsection{Motivic measure}\label{mm}
The relation
of Theorem \ref{mt} with motivic integration is the following.
When  $S$ is equal to  $h [0, 0, 0]$, the final object of
$\Def_k$, one writes ${\rm I} \ccC_+ (Z)$ for
 ${\rm I}_S \ccC_+ (Z)$ and we shall say integrable for $S$-integrable,
and similarly for  $\ccC$.
Note that  ${\rm I} \ccC_+ (h [0, 0, 0]) = \ccC_+ (h [0, 0, 0]) =
SK_0 (\RDef_k) \otimes_{\bN [\bL - 1]} A_+$ and that
${\rm I}\ccC (h [0, 0, 0]) =  K_0 (\RDef_k) \otimes_{\bZ [\bL]} A$.
For $\varphi$ in  ${\rm I} \ccC_+ (Z)$, or in  ${\rm I} \ccC (Z)$, one defines the
motivic integral
$\mu (\varphi)$ by $\mu (\varphi) = f_! (\varphi)$ with
$f$ the morphism  $Z \rightarrow h [0, 0, 0]$.

Let $X$ be in $\Def_k$ of dimension $d$.
Let
$\varphi$ be a function in
$\cC_+ (X)$, or in $\cC (X)$. We shall say
$\varphi$ is integrable if its class
$[\varphi]_d$ in
$\ccC^d_+ (X)$, resp. in $\ccC^d (X)$, is integrable,
and we shall set
$$
\mu (\varphi) = \int_X \varphi \, d\mu = \mu ([\varphi]_d).
$$


Using the following Change of Variables Theorem \ref{mcvf}, one may develop the
integration on {global} (non affine) objects endowed with a differential form of top degree
(similarly as in the $p$-adic case), cf. \cite{cl}.

\begin{thm}[Cluckers-Loeser \cite{cl}]\label{mcvf}
Let $f : Y \rightarrow X$ be an isomorphism in $\Def_k$.
For any integrable function $\varphi$ in
$\cC_+ (X)$ or $\cC (X)$,
$$
\int_X \varphi d\mu =
\int_Y  \bL^{ -\ord \jac (f)} f^* (\varphi) d\mu.
$$
\end{thm}


Also, the construction we outlined of the motivic measure
carries over almost literally to a relative setting:
one can develop a {relative} theory of motivic integration:
integrals depending on parameters of functions in $\cC_+$ or $\cC$ still belong to
$\cC_+$ or $\cC$ as functions of these parameters.

More specifically,
if $f : X \rightarrow \Lambda$ is a morphism and $\varphi$ is a function in
$\cC_+ (X)$ or $\cC (X) $ that is
 {relatively integrable} (a notion defined in \cite{cl}),
 one constructs in \cite{cl} a function
 \begin{equation}\label{relativemu}
 \mu_{\Lambda}(\varphi)
 \end{equation}
in $\cC_+ (\Lambda)$, resp.  $\cC (\Lambda)$, whose restriction to every
fiber of $f$ coincides with the integral of $\varphi$ restricted to that fiber.

 \subsection{The transfer principle}\label{tp}
We are now in the position of explaining how motivic integrals
specialize to $p$-adic integrals and may be used  to obtain  a general transfer principle
allowing to transfer relations between integrals from $\bQ_p$ to $\bF_p ((t))$ and vice-versa.

We shall assume from now on that
$k$ is a number field
with ring of integers $\cO$.
We denote by
$\cA_{\cO}$ the set of  $p$-adic completions
of  {all finite extensions} of  $k$ and by
$\cB_{\cO}$ the set of all local fields of characteristic  $>0$ which are  $\cO$-algebras.

For  $K$ in  $\cC_{\cO}:=\cA_{\cO}\cup \cB_{\cO}$, we denote by
\begin{itemize}
\item $R_K$ the valuation ring
\item $M_K$
the maximal ideal
\item $k_K$  the residue field
\item $q(K)$ the cardinal of  $k_K$
\item $\varpi_K$ a uniformizing parameter of $R_K$.
\end{itemize}
\medskip

There exists a unique morphism  $\ac:K^\times \to
k_K^\times$ extending  $R_K^\times\to
k_K^\times$ and sending $\varpi_K$ to $1$. We set
$\ac(0)=0$.
For  $N>0$, we denote by  $\cA_{\cO, N}$ the set of fields  $K$ in $\cA_{\cO}$ such that
$k_K$ has characteristic  $> N$, and similarly for  $\cB_{\cO, N}$ and  $\cC_{\cO, N}$.
To be able to interpret our formulas to fields in
$\cC_\cO$, we {restrict} the language $\cL_{\text{DP}}$ to the sub-language
$\LO$
for which coefficients in the valued field sort are assumed to belong to the subring
$\cO [[t]]$ of $k ((t))$.
We denote by  $\Def(\LO)$ the sub-category of $\Def_k$
of objects definable in
$\LO$, and similarly for functions, etc.
For instance, for
$S$ in $\Def(\LO)$, we denote by  $\cC (S, \LO)$ the ring of constructible functions on $S$
definable in $\LO$.

We consider $K$ as a  $\cO [[ t]]$-algebra via
\begin{equation}
\lambda_{\cO,K}:
\sum_{i\in\bN}a_it^i \longmapsto
\sum_{i\in\bN}a_i\varpi_K^i.\end{equation}
Hence, if  we interpret
$a$ in $\cO [[t]] $ by
$\lambda_{\cO,K}(a)$, every  $\LO$-formula $\varphi$
defines for  $K$ in $\cC_\cO$ a subset  $\varphi_K$ of some
$K^m \times k_K^n \times \bZ^r$.
One proves that if two $\LO$-formulas  $\varphi$ and
$\varphi'$ define the same subassignment  $X$ of
$h[m,n,r]$, then  $\varphi_K=\varphi'_{K}$ for
$K$ in  $\cC_{\cO,N}$ when $N \gg 0$.
This allows us to denote by $X_K$ the subset defined by $\varphi_K$,
for
$K$ in  $\cC_{\cO,N}$ when $N \gg 0$.
Similarly, every $\LO$-definable morphism
 $f:X\to Y$ specializes to $f_K:X_K\to Y_K$  for
$K$ in  $\cC_{\cO,N}$ when $N \gg 0$.

{We now explain how $\varphi$ in
$\cC(X,\LO)$ can be specialized to $\varphi_K : X_{K}\to \bQ$
for
$K$ in  $\cC_{\cO,N}$ when $N \gg 0$.}
Let us consider
$\varphi$ in  $K_0(\RDef_X (\LO))$ of the form  $[\pi:W\to X]$ with   $W$ in
$\RDef_X(\LO)$. For
$K$ in  $\cC_{\cO,N}$ with  $N \gg 0$, we have
$\pi_K:W_K\to X_K$, so we may define
 $\varphi_K:X_K\to\bQ$ by
 \begin{equation}x \longmapsto \text{card} \left(\pi_K^{-1}(x)\right).\end{equation}
 For  $\varphi$ in  $\cP(X)$, we specialize  $\bL$ into
$q_K$ and $\alpha:X\to\bZ$ into $\alpha_K:X_K\to\bZ$.
By tensor product we get
$\varphi \mapsto \varphi_K$ for
$\varphi$ in $\cC(X,\LO)$.
Note that, under that construction, functions in $\cC_+(X,\LO)$ specialize
into non negative functions.

Let  $K$ be in   $\cC_\cO$ and  $A$ be a subset of  $K^m\times
k_K^n \times \bZ^r$. We consider  the  Zariski closure
$\bar A$ of the  projection of $A$ into $\bA_K^m$. One defines a measure
$\mu$ on
$A$ by restriction of the product of the canonical
(Serre-Oesterl\'e)
measure on $\bar A (K)$
with the counting measure on $k_K^n \times \bZ^r$.

\medskip

Fix a morphism
$f : X \rightarrow \Lambda$ in $\Def (\LO)$
and consider
$\varphi$ in
$\cC (X, \LO)$.
One can show that if $\varphi$ is relatively integrable,
then,
for $N \gg 0$, for every $K$ in $\cC_{\cO,N}$, and
{for every}  $\lambda$ in $\Lambda_K$, 
the restriction $\varphi_{K,  \lambda}$ of  $\varphi_{K}$ to
$f_K^{-1}(\lambda)$ is
 integrable.

We denote by
$\mu_{\Lambda_K}(\varphi_{K})$ the function on  $\Lambda_K$ defined by \begin{equation}\lambda \longmapsto \mu (\varphi_{K, \lambda}).\end{equation}

The following theorem says that motivic  integrals
specialize to the corresponding integrals over local fields
of high enough residue field characteristic.

\medskip
\begin{thm}[Specialization, Cluckers-Loeser \cite{crexp}
\cite{clexp}]\label{specializ}
 Let $f : S \rightarrow \Lambda$ be a
morphism in $\Def (\LO)$. Let $\varphi$ be in $\cC (S, \LO)$ and
relatively integrable with respect to $f$. For $N \gg 0$, for every
$K$ in
$\cC_{\cO,N}$, 
we have
\begin{equation}
\left(\mu_\Lambda(\varphi)\right)_{K}
=
\mu_{\Lambda_K}(\varphi_{K}).\end{equation}
\end{thm}

We are now ready to state the following abstract transfer principle:

\begin{thm} [Abstract transfer principle, Cluckers-Loeser \cite{crexp} \cite{clexp}]\label{abstract}Let $\varphi$ be in
$\cC (\Lambda, \LO)$. There exists   $N$ such that
for every $K_1$, $K_2$ in  $\cC_{\cO,N}$ with  $k_{K_1}\simeq
k_{K_2}$,
\begin{equation}\varphi_{K_1} = 0
\quad
\text{if and only if}
\quad
\varphi_{K_2} = 0.
\end{equation}
\end{thm}

Putting together the  two previous theorems, one immediately gets:

\begin{thm} [Transfer  principle for integrals with parameters, Cluckers-Loeser \cite{crexp} \cite{clexp}]\label{tran}
Let  $S \rightarrow \Lambda$ and   $S' \rightarrow \Lambda$ be morphisms in
 $\Def (\LO)$. Let  $\varphi$ and $\varphi'$ be  relatively integrable functions in   $\cC (S, \LO)$
 and
 $\cC (S', \LO)$, respectively.
 There exists
 $N$ such that
for every $K_1$, $K_2$ in  $\cC_{\cO,N}$ with  $k_{K_1}\simeq
k_{K_2}$,
\begin{equation*}
\mu_{ \Lambda_{K_1}} (\varphi_{K_1}) 
= \mu_{
\Lambda_{K_1}} (\varphi'_{K_1})  
 \quad
\text{if and only if}
 \quad
\mu_{ \Lambda_{K_2}} (\varphi_{K_2}) 
= \mu_{
\Lambda_{K_2}} (\varphi'_{K_2}).
\end{equation*}
\end{thm}

\medskip
In the special case where $ \Lambda = h [0, 0, 0]$
and $\varphi$ and $\varphi'$ are in
  $\cC (S, \LO)$
 and
 $\cC (S', \LO)$, respectively,
 this follows from previous results of Denef-Loeser \cite{JAMS}.

\begin{remark}\label{affgen}
The previous constructions and
statements may be extended directly - with similar proofs -
to the global (non affine) setting.
\end{remark}

Note that when $S = S' = \Lambda = h [0, 0, 0]$, one recovers the classical

\begin{thm} [Ax-Kochen-Er{\v s}ov \cite{AK2} \cite{Ersov}]\label{ake}
Let $\varphi$ be a first order sentence \textup{(} that is, a formula with no free variables\textup{)} in the language of rings. For almost all prime number $p$, the sentence
$\varphi$ is true in $\mathbb{Q}_p$ if and only if it is true in
$\mathbb{F}_p ((t))$.
\end{thm}

\subsection{}In view of Theorem  \ref{tran}, since
we are interested in the
behavior of integrals for sufficiently large
primes $p$, we shall assume throughout
this paper, whenever it is useful to do so,
that $p$ is sufficiently large.  {\it This remains a standing
assumption throughout this paper.}

\begin{remark}\label{OL[1/N]}
Let $N>0$ be an integer and let $\LO(1/N)$ be the language $\LO$
with one extra constant symbol to denote the rational number $1/N$.
Then the above statements in section \ref{sec2} remain valid if one
works with $\LO(1/N)$ instead of with $\LO$, where now the
conditions of big enough residue field characteristic mean in
particular that the residue field characteristic is bigger than $N$.
Indeed, $\LO(1/N)$ is a definitional expansion of $\LO$ in the sense
that both languages give exactly the same definable subassignments
and definable morphisms, hence they also yield the same rings and
semirings resp.~the same groups and semigroups of constructible
functions and constructible $\ccF$unctions.
\end{remark}

\subsection{Tensoring with $\bQ$, with $\bR$, or with $\bC$}\label{sec:tensor}

Since Arthur's weight function involves volumes (see below), it is
useful to work with $\ccF$unctions in $\ccC(X)\otimes \bR$ instead of
in $\ccC(X)$ and so on, for definable subassignments $X$. One can work
similarly to tensor with $\bQ$ or $\bC$. Once we have the direct
image operators and integration operators of Theorem \ref{mt}, of
(\ref{IC(Z)}), and of (\ref{relativemu}), this is easily done as
follows. Let $f:Z\to Y$ be a morphism in $\Def_S$. Clearly $\bQ$,
$\bR$ and $\bC$ are flat $\bZ$-modules (as localizations,
resp.~direct limits of $\bZ$, resp.~of flat modules). This allows us
to view $\rm{I}_S \ccC(Z)\otimes_\bZ \bR$ as a submodule of
$\ccC(Z)\otimes_\bZ \bR$, and similarly for relative integrable
functions. Naturally, $f_!$ extends to a homomorphism $f_!\otimes
\bR:\rm{I}_S \ccC(Z)\otimes \bR \to \rm{I}_S \ccC(Y)\otimes \bR$, and
similarly for relative integrable functions. Of course, the study of
semirings of constructible functions tensored with $\bR$ does not
have any additional value since $\bR$ contains $-1$. On the other
hand, $f^*\otimes \bR$ has a natural meaning as a homomorphism from
$\ccC(Y)\otimes \bR$ to $\ccC(Z)\otimes \bR$. In this setting, the
analogue of the Change Of Variables Theorem \ref{mcvf} clearly
remains true. A function $\varphi$ in $\cC(X,\LO)\otimes \bR$ can be
specialized to $\varphi_K : X_{K}\to \bR$ for $K$ in  $\cC_{\cO,N}$
when $N \gg 0$ as in section \ref{tp}, by tensoring with $\bR$,
where the values of $\varphi_K$ now lie in $\bR$ instead of in
$\bQ$. With this notation, the analogues in the $\otimes\bR$-setting
of Theorems \ref{specializ}, \ref{abstract}, and \ref{tran} and the
analogues of Remarks \ref{affgen} and \ref{OL[1/N]} naturally hold.

\section{Definability of Field Extensions}\label{sec:field}

\subsection{}
The Denef-Pas language does not allow us to work directly with field
extensions $E/F$.  However, field
extensions $E/F$ have proxies in the Denef-Pas language, which
are obtained by
working
through the minimal polynomial $m$ of $e\in E$ that generates $E/F$,
and an explicit
basis $\{1,x,\ldots,x^{r-1}\}$ of $F[x]/(m)=E$.  The minimal polynomial
  \begin{equation}
  x^r + a_{r-1} x^{r-1} + \cdots a_0
  \end{equation}
can be identified with its list of coefficients
\begin{equation}(a_{r-1},\ldots,a_0).\end{equation}

In this paper,
general field extensions will only appear as bound variables in
formulas.  Moreover, the degree of the extensions will always be fixed.
Thus, a statements  ``there exists a field extension
of degree $r$ such that $\ldots$''  can be translated  into the
Denef-Pas language as
``there exist $a_{r-1},\ldots,a_0$ such that $x^r+a_{r-1} x^{r-1} +\cdots+ a_0$
is irreducible and such that $\ldots$.''

After identifying the field extension with $F^r$, we can define
field automorphisms by linear maps on $F^r$ that respect the field
operations.  Thus, field automorphisms are definable, and the condition
that $E/F$ is a Galois extension is definable.

\subsection{}We will see that all of the constructions in this paper
can be arranged so that the only
field extensions that are ``free'' in a formula are unramified
of fixed degree $r$.  These can be described in a field-independent way
by a minimal polynomial $x^r - a$, where $a$ satisfies a Denef-Pas
condition that it is a unit such that $x^r-a$ is irreducible.
This construction introduces a parameter $a$ to whatever formulas involve
an unramified field extension.   The free parameter $a$ will be used
for instance when we describe unramified unitary groups.

If $m(x)=x^r-a$ defines a field extension $F_r$ of degree $r$ of $F$,
then we may represent a field extension $E$ of $F$
of degree $k$ containing $F_r$
by data $(b_{k-1},\ldots,b_0,\phi)$, where $b_i\in F$ are the
coefficients of a minimal polynomial $b(y)$
of an element generating
$E/F$, and $\phi\in M_{rk}(F)$ is a matrix giving the embedding
of $F^r\to F^k$, corresponding to
  \begin{equation}F_r = F[x]/(m(x))\to F[y]/(b(y)).\end{equation}

\section{Definability of Unramified Reductive Groups}
\label{sec:group}

\subsection{}The classification of split connected reductive groups is independent
of the field $F$.     Isomorphism classes of reductive
groups are in bijective correspondence with root data:
  \begin{equation}
  D=(X^*,\Phi,X_*,\Phi^\lor),
  \end{equation}
consisting of the character group of a Cartan subgroup, the set
of roots, the cocharacter group, and the set of coroots.  The
root data are considered up to an obvious equivalence.

In particular, we may realize each split
reductive group over $\ring{Q}$.  Fix once and for all, a rational
faithful representation $\rho=\rho_D$
   \begin{equation}
   \rho:G\to GL(V)
   \end{equation}
of each reductive group
over $\ring{Q}$, attached to root data $D$.  Fix a basis of $V$
over $\ring{Q}$.
There exists a formula $\phi_{\rho,D}$ in the Denef-Pas language
(in fact a formula in the language of rings) that describes the
zero set of $\rho(G)$.

An unramified reductive group over a local field $F$ is a quasi-split
connected reductive group that splits over an unramified extension of $F$.
The classification of unramified reductive groups is independent
of the field $F$.  Isomorphism classes of unramified reductive
groups are in bijective correspondence with pairs $(D,\theta)$,
where $D$ is the root data for the corresponding split group
and $\theta$ is an automorphism of finite order of $D$ that
preserves a set of simple roots in
$\Phi$.

The quasi-split group $G$ is obtained by an outer twist of the corresponding
split group $G^*$ as follows.
Suppose that $\theta$ has order $r$.  Let $F_r$ be the unramified
extension of $F$ of degree $r$.  Let $(B,T,\{X_\alpha\})$ be a splitting
of $G^*$, consisting of a Borel subgroup $B$, Cartan $T$, and root vectors
$\{X_\alpha\}$ all defined over $\ring{Q}$.   The automorphism $\theta$
of the root data determines a unique
automorphism of $G^*$ preserving the splitting.  We let $\theta$ denote
this automorphism of $G^*$ as well.  The automorphism $\theta$ acts
on root vectors by $\theta(X_\alpha)$ = $X_{\theta\alpha}$.

For any $A$-algebra $F$, we have that $G(A)$ is the set of fixed
points of $G(A\otimes F_r)$ under the map
   $\theta\circ\tau$,
where
  \begin{equation}\tau:G^*(A\otimes F_r)\to G^*(A\otimes F_r)
  \end{equation}
extends the Frobenius automorphism of $F_r/F$.  Through the
representation $\rho$ of $G^*$, the fixed-point condition can be expressed
on the corresponding matrix groups.  In particular, $G(F)$
can be explicitly realized
as the fixed points of a map $\theta\circ\tau$
on $\rho(G^*(F_r))\subset M_n(F_r)$, the set of $n$ by $n$ matrices
with coefficients in $F_r$.

\subsection{}To express the group $G$ by a formula in the Denef-Pas language,
we introduce a free parameter $a$ as described in Section~\ref{sec:field},
which is the proxy in the Denef-Pas language
for an unramified extension of degree $r$.  Under the identification
$F_r\to F^r$, the fixed point condition becomes a ring condition
on $M_n(F)\otimes F^r$.

We find
that for any pair $(D,\theta)$ (and fixed $\rho$),
there exists a formula
$\phi=\phi_{D,\theta}$ in the Denef-Pas language with $1+r^2+n^2 r$
free variables such that
for any $p$-adic field $F$
$\phi(a,\tau,g)$ is true exactly when $a$ is a unit such that
$x^r-a$ is irreducible over $F$, $F_r = F[x]/(x^r-a)$,
$\tau\in M_r(F)$ is a generator of the Galois group of the
field extension $F_r/F$ (under the identification $F_r = F^r$),
and $g\in M_n(F)\otimes F^r$ is identified with
an element $g\in G(F)\subset M_n(F_r)$ (under the identification
of $F^r$ with $F_r$ determined by $a$).

We stress that this construction differs from the usual global to
local specialization.  If we take a unitary group defined over $\ring{Q}$
with conjugation from a quadratic
extension $\ring{Q}(\sqrt{b})/\ring{Q}$, then the corresponding
$p$-adic group at completions of $\ring{Q}$ can be split or inert depending
on the prime.  By making $a$ (or $b$) a parameter to the formula,
we can insure that the formula $\phi_{D,\theta}$ defines the
unramified unitary
group {\it at every finite place}, and not merely at the inert primes.

By a similar construction, we obtain formulas in the Denef-Pas
language for unramified reductive Lie algebras.  If the group splits
over a non-trivial unramified extension, there will be corresponding
parameters, $a$ and $\tau$.

\section{Galois Cohomology}\label{sec:gal}

\subsection{}Once we have expressed field extensions and Galois groups
within the Denef-Pas language, we may do some rudimentary
Galois cohomology.

We follow various conventions when working with Galois
cohomology groups.  We never work directly with the cohomology
groups.  Rather, we represent each class in a cohomology
group $H^r(\op{Gal}(E/F),A)$ explicitly as a cocycle
\[b \in Z^r(\op{Gal}(E/F),A),\] viewed as a tuple of elements
of $A$, indexed by $\op{Gal}(E/F)$.
We express that two cocycles $b,b'$ are cohomologous by
means of an existential quantifier: there exists a coboundary
$c$ such that $b = b' c$ (as tuples indexed by the Galois group).

The module $A$ will always
be one that whose elements we can express directly in the
Denef-Pas language.  For example, $A$ may be
 a free $\ring{Z}$-module of finite rank,
a definable subgroup of $\op{GL}_n(E)$, or the coordinate
ring $\ring{Q}[M_n]$ of the space of $n\times n$ matrices.
Similarly, we restrict ourselves to
actions of $\op{Gal}(E/F)$ on modules $A$ that can be
expressed in our first-order language.

For example, if $T$ is a torus that splits over an extension
$E/F$ of $p$-adic fields, instead of the group $H^1(F,T)$,
we work with the group of cocycles $Z^1(\op{Gal}(E/F),T(E))$,
where $T(E)$ is represented explicitly as an affine algebraic group of
invertible $n$ by $n$ matrices.
Instead of $H^1(F,X^*(T))$, we work
with $Z^1(\op{Gal}(E/F),X^*(T))$, where $X^*(T)$
is viewed as a subset of the coordinate ring of $T$.  Even
more concretely, we may view $T$ as a closed subset of
$M_n$ for some $n$, and represent
elements of the coordinate ring of $T$ as polynomials in $n^2$
variables.\footnote{Since $X^*(T)$ is a free $\ring{Z}$-module of
finite rank, it might be tempting to represent $X^*(T)$ within
the value group sort of the Denef-Pas language.  This is
the wrong way to proceed!}
The cup-product of $b \in Z^1(\op{Gal}(E/F),T(E))$
with $p\in Z^1(\op{Gal}(E/F),X^*(T))$ into
$Z^2(\op{Gal}(E/F),E^\times)$ is the tuple whose coordinates
are $p_{\sigma'}(b_\sigma)$, where the pairing is obtained by evaluating
polynomials $p_{\sigma'}$ at values $b_\sigma$.

\subsection{}
The canonical isomorphism of the Brauer group
\begin{equation}
  H^2(F,\ring{G}_m) =\ring{Q}/\ring{Z}
\end{equation}
relies on the Frobenius automorphism of unramified field
extensions of $F$.  We avoid the Frobenius automorphism and
work with an explicitly chosen generator $\tau$ of
an unramified extension.

The Tate-Nakayama pairing
\begin{equation}
   H^1(F,T) \times H^1(F,X^*(T)) \to \ring{Q}/\ring{Z}
\end{equation}
can be translated into the Denef-Pas language as a collection
of predicates $$\tn_{\ell,k}(a,\tau,b,p)$$
with free variables
\[b \in Z^1(\op{Gal}(E/F),T(E)),\quad
p\in Z^1(\op{Gal}(E/F),X^*(T)),\] and a generator $\tau$
of an unramified field extensions $F_r/F$, defined by parameter $a$.  The extension
$E/F$ is assumed to be Galois, and $r$ must be sufficiently
large with respect to the degree of the extension $E/F$.
The collection of predicates are indexed by
$\ell,k\in\ring{N}$, with $k\ne0$.

The predicate $\tn_{\ell,k}(a,\tau,b,p)$ asserts that
the Tate-Nakayama cup product pairing of $b$ and $p$ is
the class $\ell/k\in\ring{Q}/\ring{Z}$ in the Brauer group.
In more detail, we construct the predicate by a rather literal
translation of the Brauer group isomorphism into the the
language of Denef and Pas.  We review Tate-Nakayama, to
make it evident that the Denef-Pas language is all that is
needed.  There is no
harm in assuming that $F_r\subset E$.
Let $c\in Z^2(\op{Gal}(E/F),E^\times)$
be the cup-product
of $b$ and $p$.  There exists a coboundary $d$ and a cocycle
$c'\in Z^2(\op{Gal}(F_r/F),F_r^\times)$ such that
   $c = c'' d$, where $c''$ is the inflation of $c'$ to
$\op{Gal}(E/F)$.  Set $A=Z^2(\op{Gal}(F_r/F),\ring{Z})$.  Let $\val(c')\in A$
be the tuple obtained by applying the valuation to each coordinate.
Associated with the short exact sequence
    \begin{equation}
    1\to \ring{Z}\to \ring{Q}
    \to \ring{Q}/\ring{Z}\to 1
    \end{equation}
          is a connecting homomorphism
          \begin{equation}\op{Hom}(\op{Gal}(F_r/F),\ring{Q}/\ring{Z}) =
          Z^1(\op{Gal}(F_r/F),\ring{Q}/\ring{Z}) \to A.\end{equation}
         Let $e(\ell,k)$ be the image in $A$
           of the homomorphism from $\op{Gal}(F_r/F)$ to
           $\ring{Q}/\ring{Z}$ that sends the chosen generator
           $\tau$ to $\ell/k$.  (The predicate $\tn_{\ell,k}$ is defined to be
           false if $k$ does not divide $r \ell$.)
Finally, the predicate $\tn_{\ell,k}$
asserts that  $e(\ell,k) = \val(c') \in A$.

\section{Weights}

\subsection{The weight function}

Weighted orbital integrals can also be brought into
the framework of constructible functions.

Let $G$ be a connected reductive group over a $p$-adic field $F$ and let
$M$ be a Levi subgroup of $G$.  Let $\cP(M)$ be the set of
parabolic subgroups $P$ of $G$ that have a Levi decomposition $P=M_P N_P$
with $M=M_P$.

Arthur
defines a real-valued
weight function $w_M(x)=w_M^G(x)$ on the group $G(F)$.
We recall the general form of this function.  The function
is defined as the value at $\lambda=0$ of a smooth function
  \begin{equation}
    \label{eqn:vM}
  v_M(x,\lambda) = \sum_{P\in \cP(M)} v_P(x,\lambda)\theta_P(\lambda)^{-1},
  \end{equation}
whose terms are indexed by $P\in \cP(M)$.

The parameter $\lambda$ lies in a finite-dimensional
real vector space $i\a_M^*$, where
$\a_M^*$ is the dual of
  \begin{equation}\a_M=\op{Hom}(X(M)_{rat},\ring{R}),
  \end{equation}
and where $X(M)_{rat}$ is the group of $F$-rational characters of $M$.
The function
   \begin{equation}\theta_P(\lambda)\in S^p[\a_{M,\ring{C}}]\end{equation}
is a nonzero homogeneous polynomial
of some degree $p$.  The degree $p$ of this form
is independent of $P\in\cP(M)$.

The factor $v_P(x,\lambda)$ is defined as follows.  We may
assume a choice of a hyperspecial maximal compact subgroup $K$
of $G(F)$ such that is admissible in the sense of Arthur
\cite[p.9]{artIF}.  Then the Iwasawa decomposition takes
the form
  \begin{equation}
  G(F) = N_P(F) M_P(F) K.
  \end{equation}
Then for $x= n m k\in G(F)$,
set $H_P(x) = H_{M}(m)$,
where $H_M(m)$ is defined by the condition
  \begin{equation}\label{eqn:HM}
  {\langle H_M(m),\chi\rangle} = -
   \val\chi(m)
  \end{equation}
for all $\chi\in X(M)_{rat}$.
The function $v_P(\lambda,x)$ is then defined to be
   \begin{equation}
   v_P(\lambda,x) = e^{-\lambda (H_P(x))}.
   \end{equation}
The function $v_M(\lambda,x)$ defined by Equation~\ref{eqn:vM}
is not obviously a smooth function of $\lambda\in i\a_M^*$, because
the individual summands $v_P(\lambda,x)/\theta_P(\lambda)$ do not
extend continuously to $\lambda=0$.
But according to a theorem of Arthur, it is smooth.

Arthur gives the function in the following alternative form.
Fix any
generic $\lambda$.
Let $t$ be a real parameter.
The denominator is homogeneous of degree $p$, so
$\theta_P(t\lambda) = t^p \theta_P(\lambda)$.
We compute the limit of
$v_M(t \lambda,x)$ as $t$ tends to zero by applying
l'H\^opital's rule $p$ times.  The result is
   \begin{equation}\label{eqn:vMlimit}
   v_M(x) = \sum_{P\in\cP}\frac{(-1)^p (\lambda (H_P(x)))^p}{p! \theta_P(\lambda)}.
   \end{equation}
The right-hand side appears to depend on $\lambda$, but in fact it is constant
as function of $\lambda$.

\subsection{Weights and constructible functions}
Let us
recall our context for unrami\-fied groups.  Associated to root
data $D$ and an automorphism $\theta$ of $D$ that preserves positive
roots, there is a formula $\phi_{D,\theta}(a,\tau,g)$ in
$1+r^2 + n^2 r$ variables where $a$ determines an unramified
field extension $F_r/F$ of degree $r$,
$\tau$ is a generator of the Galois
group of $F_r/F$, and $g$ is an element of the quasi-split group
determined by the root data $D$, automorphism
$\theta$, field
extension $F_r/F$, and $\tau$.  This is a definable
subassignment
$\tilde G=\tilde G_{D,\theta}$.

For each $p$-adic field $F$, $a$ and $\tau$, the definable
subassignment $\tilde G = \tilde G_{D,\theta}$,
gives a reductive group $G$ and a hyperspecial
maximal compact subgroup $K$ (by taking the integer points of $G$).

We pick Levi factors in standard position in the usual way.
We fix a splitting $(B,T,\{X_\alpha\})$ of the split group $G$
as in Section~\ref{sec:group}.  We may take parabolic subgroups
$P$ containing $B$ and Levi factors $M$ generated by $T$ and
a by subset of the roots vectors $\{X_\alpha\mid\ \alpha\in S\}$,
with
$S$ a subset of the set of simple roots.
For each subset $S$ of simple roots, we have
a definable subassignment $\tilde M_S \subset \tilde G$, defined by
a formula $\phi_{S,D,\theta}(a,\tau,m)$, with $a$ and $\tau$
as before, and $m$ constrained to be
an element of the Levi factor $M_{S,a,\tau}(F)$
of the reductive group $G_{a,\tau}$ attached to $D,\theta,a,\tau$.
If $S$ is the set of all simple roots, then $\tilde M_S = \tilde G$.

\begin{lemma}\label{lemma:u}
Let $(D,\theta)$ be root data and an automorphism
as above.  For each subset $S$ of the simple roots,
There is a constructible function $u_S$ on $\tilde
G_{D,\theta}$ such
that for for any $p$-adic field $F$,
   \begin{equation}u_S(a,\tau,g) = v_M(g)
   ,\end{equation}
where $M=M_{S,a,\tau}$.
\end{lemma}

\begin{proof} Constructible functions form a ring.  Therefore
if we express $u_S$ as a polynomial 
in functions that are known
to be constructible, then it follows that $u_S$ itself is constructible.

The character group $X(M)_{rat}$
is independent of the field $F$.
The coefficients
of the form $\theta_P$ 
are independent of the field $F$.
We may compute $v_M$ from Equation~\ref{eqn:vMlimit} with respect to any sufficiently generic $\lambda$.
In particular, we may choose $\lambda\in X(M)_{rat}$,
once for all fields $F$.    Then
$\theta_P(\lambda)$ is a non-zero 
number that does not depend on $F$.

By definition,
the function $
H_M$ on $M(F)$
is a vector of valuations of polynomial expressions
in the matrix coefficients
of $m$.
The linear form $\lambda(
H_M(m))$ is then clearly a constructible
function.

Consider the function $\lambda(
H_P(g)) = \lambda(
H_M(m))$, where $g=n m k$.
If we add the parameters, $a,\tau$,
its graph is
   \begin{equation}
   \{((a,\tau,g),\ell)\in \tilde G\times\ring{Z}
   \mid \exists \, k\in K,\ m\in M, n\in N_P, \
       g = n m k,\ \ell = \lambda(
       H_M(m))\}.
   \end{equation}
This is a definable subassignment.  Hence,
 \begin{equation}(a,\tau,g)\mapsto \lambda(
 H_P(g)),
 \quad P = P_{a,\tau}\end{equation} is a constructible
function.

As the function $v_M(g)$ is constructed as a polynomial
in $\lambda(H_P(g))$ with rational coefficients, as $P$ ranges over $\cP(M)$,
it is now clear
that it can be lifted to a constructible function $u_S(a,\tau,g)$, according to  the description of
constructible functions recalled in
 Section~\ref{2.3}.
\end{proof}

\section{Measures}\label{sec:mea}

\subsection{}Let $G$ be a split reductive group over $\ring{Q}$.
Let $B$ be
a Borel subgroup of $G$ and $T$ a Cartan subgroup contained in $B$.
Let $N$ be the unipotent radical of $B$ and Let $N'$ be the unipotent
radical of the Borel subgroup opposite to $B$ through $T$.
The big cell $T N N'$ is a Zariski open subset of $G$.  The Haar
measure of $G$ is the measure attached to the differential form,
expressed on the open cell as
   \begin{equation}\omega_G = d^*t \wedge dn \wedge dn'.\end{equation}
where $d^*t$, $dn$, $dn'$ are differential forms
of top degree on $T$, $N$, and $N'$, which are bi-invariant by
the actions of $T$, $N$, and $N'$,
respectively. The radical $N$ (and $N'$) can be identified with
affine spaces and differential forms $dn$ come from a choice
of root vectors for the algebra $N$.
   \begin{equation}
   dn = dx_1 \wedge \cdots\wedge dx_k.
   \end{equation}
We pick the root vectors to respect the rational structure
of $G$.  If we identify $T$ with a split torus $\ring{G}_m^r$, then
$d^*t$ takes the form of a multiplicatively invariant form
   \begin{equation}
   \frac{dt_1}{t_1}\wedge\cdots \wedge \frac{dt_r}{t_r}.
   \end{equation}

For any $p$-adic field, we obtain a Haar measure $|\omega_G|$ on $G(F)$
as the measure attached to the form $\omega_G$.

Similarly, for an endoscopic group $H$, there is a Haar measure
on $H(F)$ obtained from a similarly constructed differential form
$\omega_H$.  According to calculations of Langlands and Shelstad
of the Shalika germ associated with regular unipotent conjugacy
classes (\cite{LS}, \cite{Sh}),
the invariant measures on cosets $T\backslash G$
and $T\backslash H$ that are
be used for the fundamental lemma have the form
   \begin{equation}|d_Tt|\backslash |\omega_G|,\quad |d_Tt|\backslash |\omega_H|,\end{equation}
where $|d_Tt|$ is a Haar measure on $T$,
considered as a Cartan subgroup of both $G$ and $H$.


\begin{lemma}  There is a definable subassignment $\tilde G_{ell}$ of regular
semisimple elliptic elements of $\tilde G_{D,\theta}$.  More generally,
for each $S$ subset of simple roots, defining a standard
Levi subgroup, there is a definable subassignment $\tilde G_{S,\ell}$ of
regular semisimple elements of $\tilde G_{G,\theta}$ that are conjugate to
an element of $\tilde M_S$ and that are elliptic in $\tilde M_S$.
Similarly, there are definable subassignments
$\tilde \g_{S,\ell}$ of
regular semisimple elements of $\tilde g$.
\end{lemma}

\begin{proof}  For each $a,\tau$, the
constraints on $g$ are that it is not conjugate to a standard
Levi subgroup of $G_{a,\tau}$, that there is a field extension (of
some fixed degree $k$) over which $g$ can be diagonalized, and that
the Weyl discriminant is nonzero.  These conditions
are all readily expressed as a formula in the Denef-Pas language.
The proof for the other statements are similar.
\end{proof}

A conjugacy class
is said to be bounded if each element $\gamma$ in that conjugacy
class belongs to a compact subgroup of $G(F)$.  Arthur states
the weighted fundamental lemma (at the group level) in terms
of bounded conjugacy classes.  Similarly, we have the following
lemma:

\begin{lemma}  There is a definable subassignment $\tilde G_{bd}$ of bounded
semisimple conjugacy classes of $\tilde G_{D,\theta}$.
\end{lemma}

\subsection{}The Weyl integration formula can be used to fix normalizations
of measures.  By \cite[page 36]{waltrac}, the Weyl integration formula
takes the form
  \begin{equation}
  \begin{array}{lll}
  \int_\g f(X) dX &= \sum_{T} |W(G,T)|^{-1}\int_t |D^\g(X)|\times\\
     &\quad\int_{G/T} f (\Ad x X)dx\,dX.
  \end{array}
  \end{equation}
The measure $dX$ is the Serre-Oesterl\'e measure, which is an
additive Haar measure on $\g$.  This is compatible with the invariant
form $\omega_G$ on $G$ in the sense of \cite{cun}.
The sum runs over conjugacy classes of Cartan subgroups.  The factor
$|W(G,T)|$ is the order of the Weyl group for $T$; and $|D^\g(X)|$ is
the usual discriminant factor.  The measure $dx$ is the quotient measure
normalized by $\omega_G$ on $G$ and $dX$ on $\tx$.

Let $\cx_G = \op{spec}\,(R^G)$, where $R$ is the coordinate ring of $\g$,
and $R^G$ is the subring of $G$-invariants.
We have a morphism $\g\to \cx_G$,
coming from the inclusion $R^G\subset R$.
For $\tx\subset\g$, the morphism $\tx\to\g\to \cx_G$ is $W(G,T)$-to-$1$ on
the set of regular semisimple elements.    The fiber over a regular semisimple
element $\gamma\in\cx_G$ is the stable orbit of $\gamma$ in $\g$.  Comparing
this to the Weyl-integration formula, we see that the choice of measure on fibers
of the morphism $\g\to\cx_G$ determined by the
Serre-Oesterl\'e measures on $\g$ and on $\cx_G$
is an invariant measure $dx$.
Similar comments apply to endoscopic groups $H$ of $G$.

These observations allow us to conclude that invariant measures on stable
conjugacy classes are
compatible with the general framework of \cite{clexp}.

\section{The Langlands-Shelstad Transfer Factor}

This section presents a few facts related to the Langlands-Shelstad
transfer factor.  The Langlands-Shelstad transfer factor was
originally defined for pairs of elements in a group $G$ and a fixed
endoscopic group.  By considering the limiting behavior of this
transfer factor near the identity element of the group, we obtain
a transfer factor on the Lie algebra.

The fundamental lemma
for the Lie algebra and the Lie group are closely related.
See \cite{Hales} for the relation in the
unweighted case and the Appendix to
\cite{w-weighted} for a sketch of the relation in the full
weighted case of the fundamental lemma.

We work with the Lie algebra version of the transfer
factor.  Only small modifications would be required to work with
the transfer factors on groups.

Let $F$ be a $p$-adic field.
We fix an unramified connected reductive group $G$ over $F$ with Lie algebra
$\g$.
The $L$-group of $G$ may be written in the form
 \begin{equation}
 {}^LG = \hat G \rtimes \Gamma,
 \end{equation}
where $\Gamma$ is the Galois group of a splitting
field of $G$.  The group $\Gamma$ acts by automorphisms
of $\hat G$ that fix a splitting $(\hat T,\hat B,\{X_\alpha\})$.

\subsection{The parameter $s$ and endoscopy}

Let $(s,\rho)$
be an unramified endoscopic datum
for $G$.  The element $s$ belongs to $\hat G/Z(\hat G)$.
Taking its preimage in $\hat G$ and replacing the endoscopic
datum by another isomorphic to it,
we may take
$s\in \hat T$ to be a semisimple element,
whose connected centralizer is defined
to be $\hat H$.  We may assume that $s$ has finite order.

By definition
$\rho:\op{Gal}(F_r/F)\to \op{Out}(\hat H)$
is a homomorphism for some unramified extension
$F_r/F$ into the group of outer automorphisms
of $H$.  See \cite[sec.7]{kottwitz2} for a review of
endoscopic data.  The expository paper \cite{Ha-SFL}
gives some additional
details in the context of the fundamental lemma.

Let $T$ be a maximally split Cartan subgroup
of $H$ and $\hat T$ its dual, which we identify with the Cartan
subgroup containing $s$.
We may assume further that $s\in \hat T^{\op{Gal}(F_r/F)}$.
Let $\lie{t}$ be the Lie
algebra of $\hat T$.
The exponential short exact sequence
   \begin{equation}
   1 \to X_*(\hat T) \to \lie{t} \to \hat T \to 1,
   \end{equation}
where $\lie{t} \to \hat T$ is the exponential map $\lambda\mapsto \exp(2\pi i \lambda)$,
gives a connecting homomorphism
   \begin{equation}
   \hat T^{\op{Gal}(F_r/F)} \to H^1(F,X^*(T)).
   \end{equation}
The image of $s$ is obtained explicitly as follows.  Write
$s = \exp(2\pi i \lambda/k)$ for some $\lambda\in X_*(T)$ and $k>0$.
Then the cocycle $\mu_\sigma\in Z^1(F,X_*(\hat T))$ is given by
the equation
   \begin{equation}
   k \mu_\sigma = (\sigma(\lambda)-\lambda),\quad
   \sigma\in\op{Gal}(F_r/F).
   \end{equation}
The action of $\op{Gal}(F_r/F)$ on $X_*(T)$ comes through the
action of $\Wth$ on $X_*(T)$
by means of a homomorphism
    \begin{equation}
    \op{Gal}(F_r/F)\to \Wth
    \end{equation}
(where $W_H$ is the Weyl group of $H$).
Instead of the cocycle $\mu_\sigma$, we prefer to work with the
corresponding cocycle
\begin{equation}\mu_w\in Z^1(\Wth,X_*(\hat T))\end{equation}
\begin{equation}
\label{eqn:mu}
 k \mu_w = (w(\lambda)-\lambda),\quad
   w\in\Wth.
   \end{equation}

From the point of view of definability in the Denef-Pas language,
a complex parameter $s$ is problematical.  However, the parameter
$\lambda\in X^*(T)$ presents no difficulties, so we discard $s$
and work directly with $\lambda$ and the fixed natural number $k$.
We consider $X^*(T)$ as
within the coordinate ring of $T$.
Through an explicit
matrix representation of $H$, for any
$\mu\in X^*(T)$, we may choose
a polynomial $P_\mu$ representing $\mu$
in the coordinate ring
of $n\times n$ matrices.  The polynomials
satisfy $P_\mu P_{\mu'} = P_{\mu+\mu'} \ \mod I$,
where $I$ is the ideal
of $T$.

Equation~\ref{eqn:mu} can be rewritten as
a collection of polynomial identities
\begin{equation}\label{eqn:poly}
  P_{w}^k = P_{w\lambda}P_{-\lambda}
  \mod I,\quad
 w\in\Wth,
\end{equation}
where $P_w = P_{\mu_w}$.
The collection of polynomials $\{P_w\}$ then
serve as the proxy for the complex parameter $s$.

\subsection{The normalization of transfer factors}

Langlands and Shelstad defined the transfer factor, up
to a scalar factor.
There is a unique choice of scalar factor for the transfer factor
that is compatible with the fundamental lemma.
Following Kottwitz,
this normalization of the transfer factor is based on
a Kostant section associated to a regular nilpotent element.
Let $\g$ be the Lie algebra
of a split reductive group $G$.
Let $\cx_G$ be the space defined in Section~\ref{sec:mea}.
There is a natural morphism $\g\to \cx_G$.  Kostant defines a
section $\cx_G\to\g$ of this morphism \cite{Ko}.  The construction
of this section
depends on a choice $X$ of regular nilpotent element.
Kostant's construction, which is based on $\lie{sl}_2$ triples,
can easily be carried out in the context of the Denef-Pas language.
In fact, once the element $X$ is given, the construction requires
only the elementary theory of rings.

We fix a regular nilpotent element as follows.
Fix a splitting $(\b,\tx,\{X_\alpha\})$ of $\g$ defined over $\ring{Q}$.
Pick $X\in\b$ such that
   \begin{equation}
   X = \sum x_\alpha X_\alpha,
   \end{equation}
where $x_\alpha$ is a unit for every simple root $\alpha$.

\subsection{The pairing}

The Lie algebra version of the
Langlands-Shelstad transfer factor $\Delta(\gamma_H,\gamma)$
is defined
for pairs $\gamma\in\g$ and $\gamma_H\in \cx_H$.  It has the form
$q^{m(\gamma_H,\gamma)}\Delta_0(\gamma_H,\gamma)$, where $\Delta_0$
is a root of unity, or zero.  (It is defined to be zero on
the set where $\gamma_H,\gamma$ are not matching elements.)

Normalize the transfer factor so that $\Delta_0(\gamma_H,\gamma)=1$
if  $\gamma$ lies in the
Kostant section associated to $X$.
By \cite{Hales}, this normalization is
independent of the choice of such $X$ (for sufficiently large
primes, as usual).  Thus, we may take the
quantifier over $X$ to be either an existential
or a universal quantifier, ranging over all such
nilpotent elements.

Let
$\tilde \cx_H\times\tilde\g_{D,\theta}$ be the definable subassignment with
free variables $(a,\tau,\gamma_H,\gamma)$ with defining condition
that $\gamma$ belongs to the Lie algebra
of $\tilde G_{D,\theta,a,\tau}$ and $\gamma_H$ belongs to the
quotient $\cx_H$ for the corresponding Lie algebra of its endoscopic group $H_{a,\tau}$.

\begin{lemma}
There is a constructible function on $\tilde \cx_H\times\tilde\g_{D,\theta}$
that specializes to the function
\begin{equation}
  q^{m(\gamma_H,\gamma)}.
\end{equation}
\end{lemma}

\begin{proof}  This is trivial.
The constructible function $\ring{L}$ specializes to $q$.
The ring of constructible functions
contains functions of the form $\ring{L}^m$, with
  \begin{equation}m:\tilde \h\times\tilde\g_{D,\theta}\to\ring{Z}\end{equation}
definable.  The function $m$ is constructible, because it is
the valuation of a polynomial in the matrix coefficients of $\gamma$
and $\gamma_H$.
\end{proof}

Let $\k$ be a lattice in $\g$ corresponding to a hyperspecial
maximal compact subgroup of $G$.  Let $\k_H$ be such a lattice in $H$.
Let $\gamma$ be a regular semi-simple element in $\g$.
Assume that it is the image of some $\gamma_H\in\k_H$.
Let $\gamma_0\in\g$
be an element in the Kostant section with the same
image as $\gamma$ in $\lie{c}$.  Let
$\inv(\gamma_0,\gamma)$ be the invariant attached to $\gamma$ and
$\kappa$ the character defined by the endoscopic data so that
$\Delta_0(\gamma_H,\gamma)=\langle\op{inv}(\gamma_H,\gamma),\kappa\rangle$.
The construction depends on
an element $g\in G(\bar F)$
such that $\Ad\,g(\gamma)=\gamma_0$.

\subsection{Description of the transfer factor}

We arrive at the following statements that summarize the value of
the transfer factor for elliptic unramified transfer factors.
The description we give is rather verbose.   In particular,
we write out the coboundaries explicitly rather than taking
classes in cohomology.   This is intentional
to prepare us for the proof that the transfer factor can be
expressed in the Denef-Pas language.

Recall that a natural number $k$ is used to define the complex
parameter $s$ in terms of $\lambda$.  The value of the transfer
factor is a $k^{\op{th}}$ root of unity. For $\ell\in\ring{N}$,
let $D_\ell(a,\tau,\gamma_H,\gamma_G)$
be the set of elements in $\h\times \g$ such that
$\gamma_H$ is a $G$-regular semisimple element of $\cx_H$,  $\gamma_G$ is
a regular semisimple element of $\g$,
and $\Delta_0(\gamma_H,\gamma_G)=\exp(2\pi \ell/k)$.

We fix a natural number $r$ that is large enough that the unramified
extension $F_r$ of $F$ of degree $r$ splits $G$ and $H$,  etc.

We fix a
natural number $N=N_r$ that is large enough that for every Cartan subgroup
in $G$ there exists a field extension of degree dividing $N$ containing
$F_r$ that splits
the Cartan subgroup.  As we are only interested in the tame situation,
we may assume that the residue characteristic $p$ does not divides $N$.

Let $T_H\subset G$
be the fixed Cartan subgroup of $G$, obtained by transfer of
the maximally split Cartan subgroup of $H$.
$\Delta_0(\gamma_H,\gamma)=0$ if and only if
$\gamma$ and $\gamma_H$ do not have the same image in $\cx_G$.

The set $D_\ell(a,\tau,\gamma_H,\gamma)$
is described by the following list of conditions.  The
parameters $a$ and $\tau$ are as above.

\begin{enumerate}
  \item Let $\gamma_0$ be the element in $\g$, constructed as the
    transfer of $\gamma_H$ to $\g$, lying in the Kostant section we
    have fixed.
  \item There exists a Galois field extension of degree $E/F$ of degree
     $N$ that contains a subfield isomorphic to $F_r$.
  \item There exists $g\in G(E)$ such that $\Ad\,g\, (\gamma) = \gamma_0$.
    Let $t_\sigma = g^{-1}\sigma(g)\in T_0(E)$, for $\sigma\in\Gal(E/F)$,
    where $T_0$ is the centralizer of $\gamma_0$.
  \item There exists $h\in G(E)$ such that  $\ad(h) T_0 = T_H$.  Let $t'_\sigma\in Z^1(E/F,T^*_H)$
  be the cocycle $\ad(h) t_\sigma$, with
  a twisted action of $\op{Gal}(E/F)$
  on $T_H$ obtained by transporting the action
  of $\op{Gal}(E/F)$ to $T_H$ via $\ad(h)$.
  The general properties of endoscopy imply
  that the cocycle $Z^1(F_r/F,X^*(T_H))$,
  defined by the collection of polynomials $P_w$, also yields a cocycle $\mu_\sigma\in
  Z^1(E/F,X^*(T^*_H))$ for this twisted action.
  \item The predicate
   $\tn_{\ell,k}(a,\tau,t'_*,\mu_*)$ of Section~\ref{sec:gal}
   expressing the Tate-Nakayama duality holds.
\end{enumerate}

This description of the transfer factor is equivalent
to descriptions found elsewhere.  We have changed the
presentation slightly by working with the polynomials $P_w$
rather than the complex parameter.
Beyond that, our description is essentially the standard one.  We have written the transfer
factor in this form to make it apparent that
nothing beyond first-order logic, basic ring arithmetic, and a
valuation map are required.

We have provided justification for the following statement.
It relies on the fixed natural number $k$ that is part of our
setup.

\begin{lemma}\label{lemma:D}  There is a definable subassignment
$D_\ell(a,\tau,\gamma_H,\gamma_G)$ with the following
interpretation.  The parameter $a$ defines an unramified
field extension $F_r$ of degree $r$, and $\tau$ is a generator
of $\op{Gal}(F_r/F)$.  The parameter $\gamma_H$ ranges over
$G$-regular elements in $\lie{c}_H = \h^H$
of the Lie algebra $\h$ of the endoscopic
group $H_{a,\tau}$.  The parameter $\gamma_G$ is a regular
semisimple element in the
Lie algebra $\g$ of the reductive group $G_{a,\tau}$.
The elements $\gamma_H$ and $\gamma_G$ are matching elements,
and the transfer factor takes the value
  $\Delta_0(\gamma_H,\gamma_G) = \exp(2\pi i\ell/k)$.
\end{lemma}

We remark that the definition of $D_\ell$
can be described in a smaller language with
two sorts, one for the valued field and another
for the value group.  The residue field sort
and the angular component map do not enter
into the description of $D_\ell$.

Moreover,
we have only made light use of the valuation.
It is used once to fix the choice of nilpotent
element that is used to construct the Kostant
section.  The valuation is used once again
in the predicate $\tn_{\ell,k}$ when working
with the Brauer group.

\begin{proof}  This follows directly from the description
of the transfer factor in terms of ring operations and
quantifiers, as provided.
\end{proof}

\subsection{Weighted orbital integrals}

Let $G$ be an unramified reductive group with hyperspecial
maximal compact subgroup $K$.  Let $M$ be a Levi subgroup
that is in good position with respect to $K$.
Let $M'$ be an unramified elliptic endoscopic
group of $M$.
Write $\g$, $\k$, $\m$, and so forth, for the corresponding
Lie algebras.   Let $1_K$ be the normalized unit of the Hecke algebra.
The normalization appropriate for our choice of measures
is discussed in \cite{Ha-SFL}.

Let $\ell'\in \cx_{M'}$ be a $G$-regular element.  
The image $\ell_M$ of $\ell'$ in $\cx_{M}$ determines a stable conjugacy class $C_M(\ell')$
in $\m$, given as the fiber in $\m$ over $\ell_M\in\cx_{M}$.
The image $\ell_G$ of $\ell'$ in $\cx_G$ gives determines a stable
conjugacy class $C_G(\ell')$ in $\g$.    Any element $x\in C_G(\ell')$
permits a representation
   \begin{equation}
     x = \Ad\,g\,\gamma
    \label{eqn:rep}
   \end{equation}
with $g\in G(F)$ and $\gamma\in C_M(\ell')$.  The map
\begin{equation}x \mapsto (\gamma_x,g_x)=(\gamma,g)\end{equation}
is well-defined up to conjugacy in $M(F)$:
\begin{equation}
(\gamma_x,g_x)\mapsto (\Ad\,m\,\gamma_x,g_x m^{-1}).
\end{equation}
In particular, the function
\begin{equation}
x \mapsto \Delta_{M',M}(\ell',\gamma_x) v_M(g_x) 1_K(x)
\label{eqn:integrand}
\end{equation}
is well-defined, where $\Delta_{M',M}$ is the normalized transfer
factor for reductive group $M$ and endoscopic group $M'$. We have
already established the constructibility of $\Delta$ and $v_M$.  The
$\ccF$unction~\ref{eqn:integrand} involves no more than $\Delta,v_M$
and additional existential quantifiers of the valued field sort,
corresponding to the existence of the representation of
Equation~\ref{eqn:rep}.   In particular,
$\ccF$unction~\ref{eqn:integrand} is constructible.  (Here and
elsewhere, we are slightly loose in our use of the term
constructible.  The precise sense of constructibility is stated in
Lemmas~\ref{lemma:u} and~\ref{lemma:D}.)

Define the weighted orbital integral to be
\begin{equation}
J^G_{M,M'}(\ell') =
\int_{C_G(\ell')} \Delta_{M',M}(\ell',\gamma_x) v_M(g_x) 1_K(x),
\end{equation}
where the integral is with respect to the quotient of
the Serre-Oesterl\'e measures on $\m$ and $\cx_{M}$.

\section{The statement of the Fundamental Lemma}

\subsection{}In \cite[Conj.~5.1]{art}, Arthur conjectures a weighted
form of the fundamental lemma.  In this section,
we review his formulation of the conjecture.
The weighted case includes the
standard fundamental lemma of Langlands and Shelstad
as a special case.
(For the formulation
of the fundamental lemma in the twisted weighted context,
see Arthur's appendix to \cite{Wh}.)

We continue
to work with the Lie algebras of endoscopic groups rather
than the endoscopic groups themselves, although this makes
very little difference for our purposes.

Our work is now almost complete.  We have already
described all of the constituents of the fundamental lemma.

Let $G$ be an unramified reductive group with hyperspecial
maximal compact subgroup $K$.  Let $M$ be a Levi subgroup
that is in good position with respect to $K$.
Let $M'$ be an unramified elliptic endoscopic
group of $M$.
Write $\g$, $\k$, $\m$, and so forth, for the corresponding
Lie algebras.

To state the weighted fundamental lemma, we need
the following additional data.
For each $G,M,M'$, there is a set of endoscopic
data ${\mathcal E}_{M'}(G)$, defined in \cite[Sec.~4]{art98}.
The set ${\mathcal E}_{M'}$ is defined by data in the dual group that
is independent of the $p$-adic field.  (The action of the Galois group
$\Gamma$ on the dual group data can be replaced with the action of
the automorphism $\theta$.)
For each, $G'\in {\mathcal E}_{M'}(G)$, there is a rational number
number $\iota_{M'}(G,G')\in\ring{Q}$, that is independent of the
$p$-adic field.

We define a function $s_M^G(\ell)$ recursively.  Assume that $s_M^{G'}$
has been defined for all $G'$ (with Levi subgroup $M$) such that
$\dim\,G' < \dim\,G$.  Then, set
\begin{equation}
   s_M^G(\ell) = J_{M,M}^G(\ell) - \sum_{G'\ne G}\iota_{M}(G,G')
    s_M^{G'}(\ell).
\end{equation}
The sum runs over ${\mathcal E}_M(G)\setminus\{G\}$.  This definition
is coherent, because each group $G'\in {\mathcal E}_M(G)$ has $M$ as a Levi
subgroup, so that $s^{G'}_M$ is defined.

The conjecture of the weighted fundamental lemma
is then that for all $G,M,M'$ as above, we have
\begin{equation}
  J_{M,M'}^G(\ell') = \sum_{G'} \iota_{M'}(G,G') s_{M'}^{G'}(\ell').
\end{equation}
for all $G$-regular elements $\ell'$ in $\cx_{M'}$.
The sum on the right runs over ${\mathcal E}_{M'}(G)$.

\subsection{Constructibility}

By our preceding discussion, we see that
the integrand (Equation~\ref{eqn:integrand}) of $J_{M,M'}^G(\ell')$
comes as
specialization of a constructible function on the definable subassignment
\begin{equation}
Z = \tilde\cx_H \times_{\tilde\cx_G} \tilde\g_{D,\theta}.
\end{equation}
This constructible function depends on parameters
$a$, $\tau$, $\gamma_H\in\tilde\cx_H$, and $\gamma\in g_{D,\theta,a,\tau}$.
If we interpret this
$p$-adically, as we vary the parameter $a$ (under the
restriction that it is a unit), the situation
specializes to isomorphic groups and Lie algebras.  In particular,
the fundamental lemma holds for one specialization of $a$ if
and only if it holds for all specializations of $a$.
As we vary the generator $\tau$ of the
Galois group of the unramified field extension $F_r/F$, we
may obtain non-isomorphic data.  Different choices of $\tau$
correspond to the fundamental lemma for various
Lie algebras
 \begin{equation}
 \g_{D,\theta'}
 \end{equation}
where $\theta'$ and $\theta$ generate the same group
$\langle\theta'\rangle=\langle\theta\rangle$ of automorphisms
of the root data.  In particular, for each $\tau$, the
constructible version specializes to a version of the $p$-adic
fundamental lemma
for Lie algebras.

\subsection{The main theorem}\label{themainth}

We state the transfer principle for the fundamental
lemma as two theorems, once in the unweighted
case and again in the weighted case.  In
fact, there is no needed for us to treat
these two cases separately; they are both
a consequence of the general transfer principle
for the motivic integrals of constructible
functions given in Theorem \ref{tran}. We state them as separate theorems,
only because of preprint of Ng\^o \cite{ngo}, which
applies directly to the unweighted case of
the fundamental lemma.

\begin{thm}[Transfer Principle for the Fundamental Lemma]\label{mt1}  Let $(D,\theta)$ be given.
Suppose that the fundamental lemma holds for
all $p$-adic fields of positive characteristic
for the endoscopic groups attached to $(D,\theta')$,
as $\theta'$ ranges over automorphisms of the
root data such that $\langle \theta'\rangle=
\langle\theta\rangle$.  Then, the fundamental
lemma holds for all $p$-adic fields of characteristic zero with sufficiently large
residual characteristic $p$
(in the same context of all endoscopic
groups attached to $(D,\theta')$).
\end{thm}

\begin{thm}[Transfer Principle for the weighted Fundamental Lemma]\label{mt2}
Let $(D,\theta)$ be given.
Suppose that the
weighted fundamental lemma holds for
all $p$-adic fields of positive characteristic
for the endoscopic groups attached to $(D,\theta')$,
as $\theta'$ ranges over automorphisms of the
root data such that $\langle \theta'\rangle=
\langle\theta\rangle$.  Then, the
weighted fundamental
lemma holds for all $p$-adic fields of characteristic zero with sufficiently large
residual characteristic $p$
(in the same context of all endoscopic
groups attached to $(D,\theta')$).
\end{thm}

\begin{proof}  We have successfully represented
all the data entering into the fundamental lemma
within the general framework of identities
of motivic integrals of constructible functions.
By the transfer principle given in Theorem \ref{tran}, the fundamental lemma holds for
all $p$-adic fields of characteristic zero,
for sufficiently large primes $p$.
\end{proof}

By the main
result of \cite{Ha-RUE}, the
unweighted fundamental lemma
holds for all elements of the Hecke algebra
for all $p$, once it holds for all sufficiently
large $p$ (for a collection of endoscopic data
obtained by descent from the original data $(D,\theta)$).  Thus, in the unweighted situation,
we can derive the fundamental lemma for all
local fields of characteristic zero, without
restriction on $p$, once the fundamental lemma
is known for a suitable collection of cases
in positive characteristic.

\section{Additive characters and the relative fundamental lemma}

\subsection{Adding exponentials}It is also possible to  enlarge $\cC (X)$ to a ring
$\cC (X)^{\exp}$ also containing
{motivic analogues of exponential functions}
and to construct a natural extension of the previous theory to
$\cC^{\exp}$.

This is performed as follows in \cite{crexp} \cite{clexp}.
Let $X$ be in $\Def_k$. We consider the category $\RDefe_{X}$
whose objects are triples $(Y\to X, \xi, g)$ with $Y$ in
$\RDef_{X}$ and $\xi : Y \rightarrow h[0,1,0]$ and $g : Y
\rightarrow h[1,0,0]$ morphisms in $\Def_k$. A morphism
$(Y'\to X, \xi', g') \rightarrow (Y\to X, \xi, g)$
in $\RDefe_{X}$ is a morphism $h : Y' \rightarrow Y$ in $\Def_X$ such that
$\xi' = \xi \circ h$ and $g' = g \circ h$.
The functor sending
$Y$ in $\RDef_{X}$ to $(Y, 0, 0)$, with $0$ denoting the constant morphism
with value $0$ in $h[0,1,0]$, resp. $h[1, 0,0]$
being fully faithful, we may consider $\RDef_{X}$ as a full subcategory of
$\RDefe_{X}$. To the category $\RDefe_{X}$ one assigns a Grothendieck ring
$K_0 (\RDefe_{X})$ defined as follows. As an abelian group it is
 the
quotient of the free abelian group over symbols $[Y \rightarrow X,
\xi,g]$ with $(Y \rightarrow X, \xi,g)$ in $\RDefe_{X}$ by the
following
four
relations
\begin{equation}[Y \rightarrow X, \xi,g] = [Y' \rightarrow X, \xi',g']
\end{equation}
for $(Y \rightarrow X, \xi,g)$ isomorphic to $(Y' \rightarrow X,
\xi',g')$,
\begin{equation}
\begin{split}
[(Y \cup Y') \rightarrow X, \xi, g] & + [(Y \cap Y') \rightarrow X,
\xi_{|Y \cap Y'},g_{|Y \cap Y'}]\\ &= [Y \rightarrow X, \xi_{|Y},
g_{|Y}] + [Y' \rightarrow X, \xi_{|Y'}, g_{|Y'}]
\end{split}
\end{equation}
for $Y$ and $Y'$ definable subassignments of some $W$ in $\RDef_X$
and $\xi$, $g$ defined on $Y \cup Y'$,
 \begin{equation} [Y\to X,\xi,g+h]=[Y\to X,\xi + \overline h,g]
 \end{equation}
for $h:Y\to h[1,0,0]$ a definable morphism with $\ord(h(y)) \geq 0$ for
all $y$ in $Y$ and $\overline h$ the reduction of $h$ modulo $t$,
and
 \begin{equation}
 [Y[0,1,0]\to X,\xi+p,g]=0
 \end{equation}
when $p:Y[0,1,0]\to h[0,1,0]$ is the projection and when
$Y[0,1,0]\to X$, $g$, and $\xi$ factorize through the projection
$Y[0,1,0]\to Y$.
Fiber product endows $K_0 (\RDefe_{X})$ with a ring structure.

Finally, one defines the ring
$\cC  (X)^{\rm exp}$ of exponential constructible functions
as
$\cC  (X)^{\rm exp}:=\cC (X)\otimes_{K_0 (\RDef_{X})} K_0 (\RDefe_{X})$.
One defines similarly $\ccC (X)^{\rm exp}$ and ${\rm I}_S \ccC (X)^{\rm exp}$.

In \cite{crexp} \cite{clexp}, given $S$ in
$\Def_k$, we construct for a morphism
$f : X \rightarrow Y$ in
$\Def_S$
 a push-forward
$f_!  : {\rm I}_S \ccC  (X)^{\rm exp} \rightarrow
{\rm I}_S \ccC  (Y)^{\rm exp}$
extending
$f_!  : {\rm I}_S \ccC  (X) \rightarrow
{\rm I}_S \ccC  (Y)$ and characterized by certain natural axioms.
In particular, the construction of the measure $\mu$
and its relative version $\mu_{\Lambda}$ extend to the exponential setting.


\subsection{Specialization and transfer principle}\label{fym}
 We denote by
 $\cC (S, \LO)^{\rm exp} $ the ring of constructible functions on $X$
definable in $\LO$.
We denote by  $\cD_K$ the set of additive characters
$\psi:K\to\bC^\times$ such that  $\psi(x)=\exp((2\pi i /p)
{\rm Tr}_{k_K}
(\bar x))$
for $x\in R_K$,
with $p$ the characteristic of $k_K$,
${\rm Tr}_{k_K}$
the trace of $k_K$ relatively to its prime field and   $\bar x$ the class of  $x$ in $k_K$.

The construction of specialization explained in \ref{tp} extends
as follows to the exponential case.
Let
$\varphi$ be in $K_0(\RDef_X (\LO))^{\rm exp}$
of the form $[W,g,\xi]$. For $\psi_K$ in $\cD_K$, one specializes
$\varphi$ into the function
$\varphi_{K,\psi_K}:X_K\to\bC$ given by \[x
\mapsto \sum_{y\in \pi_K^{-1}(x)}\psi_K(g_K(y))\exp((2\pi i /p)
{\rm Tr}_{k_K}
(\xi_K(y)))\] for $K$ in
$\cC_{\cO,N}$ with $N \gg 0$.
One defines the specialization
$\varphi \mapsto \varphi_{K,\psi_K}$ for
$\varphi$ in $\cC(X,\LO)^{\rm exp}$ by tensor product.

One can show that if $\varphi$ is relatively integrable, then, for
$N \gg 0$ and every $K$ in $\cC_{\cO,N}$, {for every}  $\lambda$ in
$\Lambda_K$ and every $\psi_K$ in $\cD_K$, the restriction
$\varphi_{K, \psi_K, \lambda}$ of  $\varphi_{K, \psi_K}$ to
$f_K^{-1}(\lambda)$ is
 integrable.

We denote by
$\mu_{\Lambda_K}(\varphi_{K, \psi_K})$ the function on  $\Lambda_K$ defined by \begin{equation}\lambda \longmapsto \mu (\varphi_{K, , \psi_K, \lambda}).\end{equation}

The results in \ref{tp}
generalize as follows to the exponential setting:

\begin{thm} [Exponential specialization, Cluckers-Loeser \cite{crexp} \cite{clexp}]Let
$f : S \rightarrow \Lambda$ be a morphism in $\Def (\LO)$. Let
$\varphi$ be in $\cC (S, \LO)^{\rm exp}$ relatively integrable with
respect to $f$. For $N \gg 0$, for every $K$ in
$\cC_{\cO,N}$ and every $\psi_K$ in $\cD_K$, we have
\begin{equation}
\left(\mu_\Lambda(\varphi)\right)_{K, \psi_K}
=
\mu_{\Lambda_K}(\varphi_{K, \psi_K}).\end{equation}
\end{thm}

\begin{thm} [Exponential abstract transfer principle, Cluckers-Loeser \cite{crexp} \cite{clexp}]Let $\varphi$ be in
$\cC (\Lambda, \LO)^{\rm exp}$. There exists   $N$ such that
for every $K_1$, $K_2$ in  $\cC_{\cO,N}$ with  $k_{K_1}\simeq
k_{K_2}$,
\begin{equation}\varphi_{K_1, \psi_{K_1}} = 0 \quad \text{for all} \quad
\psi_{K_1}\in\cD_{K_1}
\quad
\text{if and only if}
\quad
\varphi_{K_2, \psi_{K_2}} = 0 \quad \text{for all} \quad
\psi_{K_2}\in\cD_{K_2}.
\end{equation}
\end{thm}

\begin{thm} [Exponential transfer  principle for integrals with parameters, Cluckers-Loeser  \cite{crexp} \cite{clexp}]\label{tranexp}
Let  $S \rightarrow \Lambda$ and   $S' \rightarrow \Lambda$ be morphisms in
 $\Def (\LO)$. Let  $\varphi$ and $\varphi'$ be  relatively integrable functions in   $\cC (S, \LO)^{\rm exp}$
 and
 $\cC (S', \LO)^{\rm exp}$, respectively.
 There exists
 $N$ such that
for every $K_1$, $K_2$ in  $\cC_{\cO,N}$ with  $k_{K_1}\simeq
k_{K_2}$,
\begin{equation*}
\begin{split}
\mu_{ \Lambda_{K_1}} (\varphi_{K_1, \psi_{K_1}}) &= \mu_{
\Lambda_{K_1}} (\varphi'_{K_1, \psi_{K_1}})   \quad \text{for all} \quad
\psi_{K_1}\in\cD_{K_1} \\
 \quad
&\text{if and only if}\\
 \quad
\mu_{ \Lambda_{K_2}} (\varphi_{K_2, \psi_{K_2}}) &= \mu_{
\Lambda_{K_2}} (\varphi'_{K_2, \psi_{K_2}})  \quad \text{for all} \quad
\psi_{K_2}\in \cD_{K_2}.
\end{split}\end{equation*}
\end{thm}

\subsection{Jacquet-Ye integrals and the relative fundamental lemma}\label{jyf}
A specific situation where  Theorem \ref{tran} applies is that of
Jacquet-Ye integrals.
Let $E / F$ be a
unramified
degree two extension of non archimedean local fields of residue characteristic $\not= 2$ and let $\psi$ be a non trivial additive character of $F$ of conductor $\cO_F$.
Let $N_n$ be the group of upper triangular matrices with $1$'s on the  diagonal
and consider the character
$\theta : N_n (F) \rightarrow \bC^{\times}$ given
by
\begin{equation}
\theta (u) := \psi (\sum_i u_{i, i + 1}).
\end{equation}
For $a$ the diagonal matrix $(a_1, \cdots, a_n)$  with $a_i$ in $F^{\times}$,
we
consider the integral
\begin{equation}
I (a) :=
\int_{N_n (F) \times  N_n (F)}  {\bf 1}_{M_n (\cO_F)} ({}^tu_1 a u_2) \, \theta (u_1 u_2) \, du_1 du_2.
\end{equation}
Here $du$ denote the Haar measure on $N_n (F)$
with the normalization
$\int_{N_n (\cO_F)} du = 1$.

Similarly, one defines
\begin{equation}
J (a) :=
\int_{N_n (E)}  {\bf 1}_{M_n (\cO_E) \cap H_n} ({}^t {\bar u} a u) \, \theta (u \bar u) \, du,
\end{equation}
with $H_n$ the set of Hermitian matrices.

The Jacquet-Ye Conjecture \cite{JY}, proved by  Ng\^o \cite{Ngo} over function fields and by Jacquet \cite{J} in general,
asserts that
\begin{equation}\label{jy}
I (a) = \gamma (a) \, J (a)
\end{equation}
with
$$
\gamma (a) := \prod_{1 \leq i \leq n - 1} \eta (a_1 \cdots a_i),
$$
and $\eta$  the
unramified
multiplicative character of order 2 on $F^{\times}$.

It should be  clear by now to the reader that the exponential version of the Transfer Theorem \ref{tranexp} applies to
(\ref{jy})
using the proxies  for field
extensions explained in section \ref{sec:field} and viewing $a$ as a parameter. 
Note that the discrepancy between the conditions on conductors in \ref{fym}
and \ref{jyf} is handled by performing an homothety of ratio $t$.
Also it is most likely that
 Theorem \ref{tranexp} can be used for  other versions of the relative fundamental lemma.


\end{document}